\documentclass[reqno]{amsart}
\usepackage{amsmath,amssymb,amsfonts,caption}
\usepackage{graphicx}
\usepackage{color}
\usepackage[usenames,dvipsnames,svgnames,table]{xcolor}
\definecolor{orange}{rgb}{1,0.5,0}
\usepackage{euscript}
\usepackage{enumerate}
\usepackage{verbatim}
\usepackage{epsfig}
\newtheorem{theorem}{Theorem}
\newtheorem{lemma}{Lemma}
\newtheorem{proposition}{Proposition}
\newtheorem{example}{Example}
\newtheorem{corollary}{Corollary}
\newcommand{\eps}{\varepsilon}
\renewcommand{\phi}{\varphi}
\newcommand{\ds}{\, ds}
\newcommand{\dr}{\, dr}
\newcommand{\du}{\, du}

\newcommand{\dint}{\displaystyle \int}
\DeclareMathOperator{\Expo}{Exp}

\DeclareMathOperator{\e}{e}
\usepackage{dsfont}
   \newcommand{\N}{\ensuremath{\mathds N}}
   \newcommand{\R}{\ensuremath{\mathds R}}
\begin{document}
%-----------------------------------------------------------------------------%
\title[A non-autonomous SEIRS model with general incidence rate]
   {A non-autonomous SEIRS model with general incidence rate}
\author{Joaquim P. Mateus}
\address{J. Mateus\\
   Instituto Polit\'ecnico da Guarda\\
   6300-559 Guarda\\
   Portugal}
\email{jmateus@ipg.pt}
\author{C\'esar M. Silva}
\address{C. Silva\\
   Departamento de Matem\'atica\\
   Universidade da Beira Interior\\
   6201-001 Covilh\~a\\
   Portugal}
\email{csilva@mat.ubi.pt}
\urladdr{www.mat.ubi.pt/~csilva}
\date{\today}
\thanks{C. Silva was partially supported by FCT through CMUBI (project PEst-OE/MAT/UI0212/2011) and J. Mateus was partially supported by FCT through UDI (project PEst-OE/EGE/UI4056/2011)}
\subjclass[2010]{92D30, 37B55} \keywords{Epidemic model,
non-autonomous, global stability}
%-----------------------------------------------------------------------------%
\begin{abstract} For a non-autonomous SEIRS model with general incidence, that admits [T. Kuniya and Y. Nakata, Permanence and extinction for a nonautonomous SEIRS epidemic model, Appl. Math. Computing 218, 9321-9331 (2012)] as a very particular case, we obtain conditions for extinction and strong persistence of the infectives. Our conditions are computed for several particular settings and extend the hypothesis of several proposed non-autonomous models. Additionally we show that our conditions are robust in the sense that they persist under small perturbations of the parameters in some suitable family. We also present some simulations that illustrate our results.
\end{abstract}
%-----------------------------------------------------------------------------%
\maketitle
%-----------------------------------------------------------------------------%
\section{Introduction}
%-----------------------------------------------------------------------------%
The study of epidemiological models has a long history that goes back to the construction of the ODE compartmental model of Kermack and Mckendrick~\cite{Kermack-McKendrick-PRSL-1927} in 1927. Since then, several aspects of these models were considered, including thresholds conditions for persistence and extinction of the disease, existence of periodic orbits, stability and bifurcation analysis.

In this work we focus on SEIRS models. For this models, several incidence functions were discussed for the contact between susceptibles and infectives and it is known that epidemiological models with different incidence rates can exhibit very distinct dynamical behaviors. In~\cite{Driessche-Hethcote-JMB-1991} Hethcote and den Driessche considered an autonomous SEIRS model with general incidence. In this paper we will consider a family of models with general incidence in the non-autonomous setting. Namely, we will consider models of the form
\begin{equation}\label{eq:ProblemaPrincipal}
\begin{cases}
S'=\Lambda(t)-\beta(t)\,\phi(S,N,I)-\mu(t) S+\eta(t)R\\
E'=\beta(t)\,\phi(S,N,I)-(\mu(t)+\epsilon(t))E\\
I'=\epsilon(t)E -(\mu(t)+\gamma(t))I \\
R'=\gamma(t)I-(\mu(t)+\eta(t))R \\
N=S+E+I+R
\end{cases}
\end{equation}
where $S$, $E$, $I$, $R$ denote respectively the susceptible, exposed (infected but not infective), infective and recovered compartments and $N$ is the total population, $\Lambda(t)$ denotes the birth rate, $\beta(t)\,\phi(S,N,I)$ is the incidence into the exposed class of susceptible individuals, $\mu(t)$ are the natural deaths,
$\eta(t)$ represents the rate of loss of immunity, $\epsilon(t)$ represents the infectivity rate and $\gamma(t)$ is the rate of recovery.

Our general non-autonomous setting allows the discussion of the effect of seasonal fluctuations but also of environmental and demographic effects that are non periodic. For instance, for some diseases like cholera and yellow fever, the size of the latency period may decrease with global warming~\cite{Shope-EHP-1991} and this type of effects lead to non-periodic parameters.

A particular case of our setting is the case of mass-action incidence, $\phi(S,N,I)=SI$, that was considered in papers by Zhang and Teng~\cite{Zhang-Teng-BMB-2007} and by Kuniya and Nakata~\cite{Nakata-Kuniya-JMAA-2010,Kuniya-Nakata-AMC-2012}. For mass action incidence, Teng and Zhang defined a condition for strong persistence and a condition for extinction based on the sign of some constants that, even in the autonomous setting, were not thresholds. To improve this result in the periodic mass action setting~\cite{Nakata-Kuniya-JMAA-2010}, Kuniya and Nakata obtained explicit conditions based in a general method developed by Wang and Zhao~\cite{Wang-Zhao-JDDE-2008} and Rebelo, Margheri and Baca\"er~\cite{Rebelo-Margheri-Bacaër-2012} and, in the general mass action non-autonomous setting, Zhang and Teng's result was improved in~\cite{Kuniya-Nakata-AMC-2012}. In this paper we follow the approach in~\cite{Kuniya-Nakata-AMC-2012} to obtain explicit criteria for strong persistence and extinction in the non-autonomous setting with general incidence and we consider particular situations, including autonomous and asymptotically autonomous models with general incidence, periodic models with general incidence and non-autonomous model with Michaelis-Menten incidence.

For non-autonomous models with no latency class~\cite{Pereira-Silva-Silva-MMAS-2013,Zhang-Teng-Gao-AA-2008} similar results were obtained. We emphasize that our situation is very different and in particular, unlike the referred papers, in general we need three conditions to guarantee extinction and other three to guarantee strong persistence.

Additionally to the obtention of strong persistence and extinction conditions, we also show that these conditions are robust in a large family of parameter functions. Namely, we show that if our conditions determine extinction (respectively strong persistence) and we replace $\beta$, $\eta$, $\epsilon$ and $\gamma$ by different parameter functions sufficiently close in the $C^0$ topology and also replace $\phi$ by some sufficiently close incidence function we still have extinction (respectively strong persistence) for the new model.

The structure of this paper is the following: in section~\ref{section:NP} we introduce some notations, our setting and state some simple facts about our system, in section~\ref{section:MR} we state our main theorems, in section~\ref{section:E} we apply Theorem~\ref{teo:Main} to particular situations including autonomous and asymptotically autonomous models, periodic models and non-autonomous models with Michaelis-Menten incidence functions, in section~\ref{section:P} we present the proofs of our results and finally, in section~\ref{section:D} we make some final comments about our results.
%-----------------------------------------------------------------------------%
\section{Notation and Preliminaries} \label{section:NP}
%-----------------------------------------------------------------------------%

We will assume that $\Lambda$, $\mu$, $\beta$, $\eta$, $\epsilon$ and $\gamma$ are continuous bounded and nonnegative functions on $\R_0^+$, that $\phi$ is a continuous bounded and nonnegative function on $(\R_0^+)^3$ and that there are $\omega_\mu,\omega_\Lambda,\omega_\beta >0$ such that
   \begin{equation}\label{eq:d-Lambda-}
   \mu_{\omega_\mu}^-> 0, \ \Lambda_{\omega_\Lambda}^- > 0 \quad \text{and} \quad \beta_{\omega_\beta}^- > 0
   \end{equation}
where we are using the notation
$$\quad h_{\omega}^-=\liminf_{t \to +\infty} \frac{1}{\omega} \int_t^{t + \omega} h(s) \ds \quad \text{and} \quad h_{\omega}^+=\limsup_{t \to +\infty} \frac{1}{\omega} \int_t^{t+\omega} h(s) \ds,$$
that we will keep on using throughout the paper. For bounded $h$ we will also use the notation
$$h_S=\sup_{t \ge 0} h(t).$$

For each $\delta$ and $\theta$ with $\delta > \theta \ge 0$ define the set
$$\Delta_{\theta,\delta}=\{(x,n,z) \in \R^3 \colon \ \theta \le x \le n \le \delta \ \wedge \ 0 \le z \le n \le \delta \}.$$
We note that for every solution $(S(t),E(t),I(t),R(t))$ of our system the vector $(S(t),N(t),I(t))$ with $N(t)=S(t)+E(t)+I(t)+R(t)$ stays in the region $\Delta_{0,K}$ (we can take any constant $K>D$ with $D$ given by~\ref{cond-3-bs}) in Proposition~\ref{subsection:BS}) for every $t \in \R_0^+$ sufficiently large.
We need some additional assumptions about our system.
Assume that:
\begin{enumerate}[$\text{H}$1)]
\item \label{cond-CC1} for each $0\le x \le K$ and $0\le z \le K$, the function $n \mapsto \phi(x,n,z)$ is non increasing, for each $0\le z \le n \le K$ the function $x \mapsto \phi(x,n,z)$ is non decreasing and for each $0\le z\le K$ the function $x \mapsto \phi(x,x,z)$ is non decreasing and $\phi(0,n,z) = 0$;
\item \label{cond-CC5} for each $0\le x \le n \le K$ the limit
        $$\lim_{z \to 0^+} \dfrac{\phi(x,n,z)}{z}$$
        exists and the convergence is uniform in $(x,n)$ verifying $0\le x \le n \le K$;
\item \label{cond-CC3} for each $0 \le x \le n \le K$, the function
       $$
       z \mapsto
       \begin{cases}
       \dfrac{\phi(x,n,z)}{z} & \text{if} \quad 0 \le z \le K\\
       \displaystyle \lim_{z \to 0^+} \dfrac{\phi(x,n,z)}{z} & \text{if} \quad z = 0
       \end{cases}$$
    is continuous, bounded and non increasing;
\item \label{cond-CC2} given $\theta>0$ there is $K_\theta>0$ such that
$$|\phi(x_1,n,z) - \phi(x_2,n,z)| \le K_\theta |x_1-x_2|z,$$
for $(x_1,n_1,z),(x_2,n_2,z)\in \Delta_{\theta,K}$, and
$$|\phi(x_1,x_1,z) - \phi(x_2,x_2,z)| \le K_\theta |x_1-x_2|z,$$
for $(x_1,x_1,z),(x_2,x_2,z)\in \Delta_{\theta,K}$.
\end{enumerate}
Note that, by~H\ref{cond-CC5}) and~H\ref{cond-CC3}) and for every $0 \le x \le n \le K$ and $0 \le z \le n \le K$, there is $M>0$ such that we have
\begin{equation}\label{cond-CC4}
\dfrac{\phi(x,n,z)}{z} \le \lim_{\delta \to \, 0^+} \frac{\phi(x,n,\delta)}{\delta} \le M <+\infty.
\end{equation}
Note also that, if for each $\theta \in ]0,K]$ there is $K_\theta>0$ such that
   $$\dfrac{\partial \phi}{\partial x} (x,n,z) \le K_\theta z,$$
for all $(x,n,z)\in \Delta_{\theta,K}$, then H\ref{cond-CC2}) holds.

We emphasise that, as we will see, conditions~H\ref{cond-CC1})--H\ref{cond-CC2}) are verified in the usual examples.

We now state some simple facts about our system.

\begin{proposition}\label{subsection:BS}
We have the following:
\begin{enumerate}[i)]
\item \label{cond-1-bs} all solutions $(S(t),E(t),I(t),R(t))$ of~\eqref{eq:ProblemaPrincipal}
with nonnegative initial conditions,
$S(0),E(0),I(0),R(0) \ge 0$,
are nonnegative for all $t \ge 0$;
\item \label{cond-2-bs} all solutions $(S(t),E(t),I(t),R(t))$ of~\eqref{eq:ProblemaPrincipal}
with positive initial conditions, $S(0)$, $E(0)$, $I(0)$, $R(0)>0$, are positive for all $t \ge 0$;
\item \label{cond-3-bs} There is a constant $D>0$ such that, if $(S(t),E(t),I(t),R(t))$ is a solution
of~\eqref{eq:ProblemaPrincipal} with nonnegative initial conditions, $S(0),E(0),I(0),R(0) \ge 0$, then
\[
\limsup_{t \to +\infty} N(t) = \limsup_{t \to +\infty} \left( S(t)+E(t)+I(t)+R(t) \right) \le D.
\]
\end{enumerate}
\end{proposition}

\begin{proof}
Properties~\ref{cond-1-bs}) and~\ref{cond-2-bs}) are easy to prove.
In fact, since $t \mapsto \Lambda(t)$ and $t \mapsto \mu(t)$ are bounded, adding the first four equations in~\eqref{eq:ProblemaPrincipal}
we obtain for nonnegative initial conditions,
\[
N'=\Lambda(t)-\mu(t) N.
\]
By~\eqref{eq:d-Lambda-}, there is $T \ge 0$ such that $\int_t^{t+\omega_\mu}\mu(s)\ds \ge \frac12 \mu_{\omega_\mu}^- \omega_\mu$ for $t \ge T$. Thus, given $t_0 \ge T$ we have
    \[
    \begin{split}
    \int_{t_0}^t\mu(s)\ds
    & \ge \int_{t_0}^{t_0+\lfloor\frac{t-t_0}{\omega_\mu}\rfloor \omega_\mu} \mu(s)\ds\\
    & \ge \frac12 \mu_{\omega_\mu}^- \omega_\mu \lfloor\frac{t-t_0}{\omega_\mu}\rfloor \\
    & \ge \frac12 \mu_{\omega_\mu}^- \omega_\mu \left( \frac{t-t_0}{\omega_\mu} -1 \right) \\
    & = \frac12 \mu_{\omega_\mu}^- (t-t_0) - \frac12 \mu_\omega^- \omega_\mu
    \end{split}
    \]
and, setting $\mu_1=\dfrac12 \mu_{\omega_\mu}^-$ and $\mu_2=\dfrac12 \mu_{\omega_\mu}^- \omega_\mu$, we conclude that there are $\mu_1,\mu_2>0$ and $T>0$ sufficiently large such that, for all $t\ge t_0 \ge T$ we have
	\begin{equation}\label{eq:infboundmu}
        \int_{t_0}^t\mu(s)\ds \ge \mu_1(t-t_0)-\mu_2.
    \end{equation}
By~\eqref{eq:infboundmu} we have, for all $t \ge T$,
   \[
   \begin{split}
   N(t)
   & = \e^{-\int_{t_0}^t\mu(s)\ds} N_0 + \dint_{t_0}^t \e^{-\int_u^t\mu(s)\ds} \Lambda(u)\du \\
   & \le \e^{-\mu_1 (t-t_0)+\mu_2} N_0+\Lambda_S \dint_{t_0}^t \e^{-\mu_1(t-u)+\mu_2}\du \\
   & = \e^{-\mu_1(t-t_0)+\mu_2} N_0+\frac{\Lambda_S \e^{\mu_2}}{\mu_1} \left(1- \e^{-\mu_1(t-t_0)}\right)
   \end{split}
   \]
Therefore
$$\limsup_{t \to +\infty} N(t) \le \limsup_{t \to +\infty} \left[ \e^{-\mu_1(t-t_0)+\mu_2} N_0+\frac{\Lambda_S \e^{\mu_2}}{\mu_1} \left(1- \e^{-\mu_1(t-t_0)}\right)\right] = \frac{\Lambda_S\e^{\mu_2}}{\mu_1}$$
and we obtain the result setting $D=\Lambda_S\e^{\mu_2}/\mu_1$.
\end{proof}

Proposition~\ref{subsection:BS} shows that, for every $\delta>0$, $K>D$ (with $D$ given by \eqref{cond-3-bs} in Proposition~\ref{subsection:BS})
and every solution $(S(t),E(t),I(t),R(t))$ of our system, $(S(t),N(t),I(t))$ stays in the region $\Delta_{0,K}$ for $t \in \R_0^+$ sufficiently large.

%-----------------------------------------------------------------------------%
\section{Main results} \label{section:MR}
%-----------------------------------------------------------------------------%

We need to consider the following auxiliar differential equation
\begin{equation}\label{eq:SistemaAuxiliar}
z'=\Lambda(t)-\mu(t)z.
\end{equation}

The next result summarizes some properties of the given equation.
\begin{proposition} \label{eq:subsectionAS}
We have the following:
\begin{enumerate}[i)]
\item \label{cond-1-aux}
Given $t_0 \ge 0$, all solutions $z(t)$ of equation~\eqref{eq:SistemaAuxiliar} with initial condition $z(t_0) \ge 0$ are nonnegative for all $t \ge 0$;
\item \label{cond-1a-aux}
Given $t_0 \ge 0$, all solutions $z(t)$ of equation~\eqref{eq:SistemaAuxiliar} with initial condition $z(t_0) > 0$ are positive for all $t \ge 0$;
\item \label{cond-2-aux} Each fixed solution $z(t)$ of~\eqref{eq:SistemaAuxiliar} with initial condition $z(t_0) \ge 0$ is bounded and globally uniformly attractive on $[0,+\infty[$;
\item \label{cond-3-aux} there is $D\ge 0$ and $T>0$ such that if $t_0\ge T$, $z(t)$ is a solution of~\eqref{eq:SistemaAuxiliar} and $\tilde z(t)$ is a solution of
\begin{equation}\label{eq:sistauxperturb}
  z'=\Lambda(t) - \mu(t) z + f(t)
\end{equation}
with $f$ bounded and $\tilde z(t_0)=z(t_0)$ then
\[
\sup_{t \ge t_0} |\tilde z(t) - z(t)| \le D \ \sup_{t \ge t_0} |f(t)|.
\]
\item \label{cond-4-aux}
There exists constants $m_1,m_2>0$ such that,
for each solution of~\eqref{eq:SistemaAuxiliar} with $z(0)=z_0>0$, we have
\[
m_1 \le \liminf_{t \to \infty} z(t) \le \limsup_{t \to \infty} z(t) \le m_2.
\]
\end{enumerate}
\end{proposition}

\begin{proof}
Given $t_0\ge 0$, the solution of~\eqref{eq:SistemaAuxiliar} with initial condition $z(t_0)=z_0$ is given by
	$$z(t)=\e^{-\int_{t_0}^t \mu(s) \ds} z_0 +\dint_{t_0}^t \e^{-\int_u^t \mu(s) \ds} \Lambda(u) \du $$
and thus, since $\Lambda(t)\ge 0$ for all $t\ge 0$, if $z_0\ge 0$ we obtain $z(t)\ge 0$ for all $t\ge t_0$ and if $z_0> 0$ we obtain $z(t)> 0$ for all $t\ge t_0$. This establishes~\ref{cond-1-aux}) and~\ref{cond-1a-aux}).

By~\eqref{eq:d-Lambda-} (recalling~\eqref{eq:infboundmu}), there are $\mu_1,\mu_2>0$ sufficiently small and $t_0>0$ sufficiently large such that, for all $t \ge t_0$ we have
	\begin{equation} \label{eq:split1}
   \begin{split}
   z(t)
   & = \e^{-\int_{t_0}^t\mu(s)\ds} z_0 + \dint_{t_0}^t \e^{-\int_u^t\mu(s)\ds} \Lambda(u)\du \\
   & \le \e^{-\mu_1 (t-t_0)+\mu_2} z_0+\Lambda_S \dint_{t_0}^t \e^{-\mu_1 (t-u)+\mu_2}\du \\
   & = \e^{-\mu_1 (t-t_0)+\mu_2} z_0+\frac{\Lambda_S \e^{\mu_2}}{\mu_1} \left(1- \e^{-\mu_1(t-t_0)}\right)
   \end{split}
   \end{equation}
and we conclude that $z(t)$ is bounded.

Let $z_1$ be a solution of~\eqref{eq:SistemaAuxiliar} with $z_1(t_0)=z_{0,1}$. By~\eqref{eq:d-Lambda-}, there is $t_0>0$ and $\tilde{\mu}>0$ such that, for $t \ge t_0$ we have
	\[
   |z(t)-z_1(t)| = \e^{-\int_{t_0}^t\mu(s)\ds} |z_0-z_{0,1}| \le \e^{-\mu_1 (t-t_0)+\mu_2} |z_0-z_{0,1}|
   \]
and thus $|z(t)-z_1(t)|\to0$ as $t \to +\infty$ and we obtain~\ref{cond-2-aux}).

Subtracting~\eqref{eq:SistemaAuxiliar} and~\eqref{eq:sistauxperturb} and setting $w(t)=\tilde{z}(t)-z(t)$ we obtain
   $$w'=-\mu(t) w + f(t)$$
and thus, since $w(t_0)=\tilde{z}(t_0)-z(t_0)=0$, we get again by~\eqref{eq:d-Lambda-} (and the computations in~\eqref{eq:infboundmu}), for $t_0$ sufficiently large
\[
\begin{split}
|\tilde{z}(t)-z(t)|
& =|w(t)| = \dint_{t_0}^t \e^{-\int_u^t \mu(s) \ds} |f(u)| \du
\le \sup_{t \ge t_0} |f(t)| \dint_{t_0}^t \e^{-\mu_1 (t-u)+\mu_2} \du \\
& = \frac{\e^{\mu_2}}{\mu_1} \ \sup_{t \ge t_0} |f(t)| \ \left(1-\e^{-\mu_1( t- t_0)}\right)
\le \frac{\e^{\mu_2}}{\mu_1} \ \sup_{t \ge t_0} |f(t)|,
\end{split}
\]

 for all $t \ge t_0$, and we obtain~\ref{cond-3-aux}).

For all $t>0$ sufficiently large there is $\Lambda_1>0$ such that
\[
\begin{split}
z(t)
& = \e^{-\int_{t-\omega_\Lambda}^t \mu(s) \ds} z_0 +\dint_{t-\omega_\Lambda}^t \e^{-\int_u^t \mu(s) \ds} \Lambda(u) \du  \\
& \ge \dint_{t-\omega_\Lambda}^t \e^{-\mu_S \omega_\Lambda} \Lambda(u) \du  \\
& \ge \Lambda_1 \e^{-\mu_S\omega_\Lambda}
\end{split}
\]
and thus $\displaystyle \liminf_{t \to +\infty} z(t) \ge \Lambda_1 \e^{-\mu_S \omega_\Lambda}$. By~\eqref{eq:split1} we have $\displaystyle \limsup_{t \to +\infty} z(t) \le \frac{\Lambda_S \e^{\mu_2}}{\mu_1}$. Therefore we obtain~\ref{cond-4-aux}).
\end{proof}

For $p>0$ and $t>0$, define the auxiliary functions
\begin{equation}\label{eq:glambda}
g_\delta (p,t,z)=\beta(t) \frac{\phi(z(t),z(t),\delta)}{\delta} \, p + \gamma(t) - \left(1+\frac{1}{p} \right) \epsilon(t),
\end{equation}	
\[
h(p,t)=\gamma(t) - \left(1+\frac{1}{p} \right) \epsilon(t),
\]	

\begin{equation}\label{eq:bSlambda}
b_\delta(p,t,z)=\beta(t)\frac{\phi(z(t),z(t),\delta)}{\delta} \, p - \mu(t) - \epsilon(t),
\end{equation}		
where $z(t)$ is any solution of~\eqref{eq:SistemaAuxiliar} such that $z(0)>0$, and also consider the function
$$W(p,t)=p E(t) - I(t).$$
 For each solution $z(t)$ of~\eqref{eq:SistemaAuxiliar} with $z(0)>0$ and $\lambda>0, p>0$ we define
\begin{equation} \label{eq:relambdap}
R_e(\lambda,p)= \Expo\left[\limsup_{t \to +\infty} \int_t^{t+\lambda} \lim_{\delta \to 0^+}
 b_{\delta}(p,s,z(s)) \ds\right],
\end{equation}
\begin{equation} \label{eq:rplambdap}
R_p(\lambda,p)= \Expo\left[\liminf_{t \to +\infty} \int_t^{t+\lambda} \lim_{\delta \to 0^+}
 b_\delta(p,s,z(s)) \ds\right],
\end{equation}
\begin{equation} \label{eq:r*elambdap}
R^*_e(\lambda,p)=\Expo\left[\limsup_{t \to +\infty} \int_t^{t+\lambda}
\frac{\epsilon(s)}{p}-\mu(s)-\gamma(s) \ds\right],
\end{equation}
\begin{equation} \label{eq:r*plambdap}
R^*_p(\lambda,p)=\Expo\left[\liminf_{t \to +\infty} \int_t^{t+\lambda}
\frac{\epsilon(s)}{p}-\mu(s)-\gamma(s) \ds\right],
\end{equation}
and finally
\begin{equation} \label{eq:Gp}
G(p)=\limsup_{t \to +\infty}\lim_{\delta \to 0^+} g_\delta(p,t,z(t))
\end{equation}
and
\begin{equation} \label{eq:Hp}
H(p)=\liminf_{t \to +\infty} h(p,t).
\end{equation}

Note that, if the incidence function is differentiable, then the equations~\eqref{eq:relambdap},~\eqref{eq:rplambdap} and~\eqref{eq:Gp} simplify. In fact, in this case, according to~H\ref{cond-CC2}) we have $\phi(x,n,0)=0$, and thus
$$\lim_{\delta \to 0^+} \dfrac{\phi(z(t),z(t),\delta)}{\delta}=\dfrac{\partial \phi}{\partial z}(z(t),z(t),0).$$

The next lemma shows that numbers $R_e(\lambda,p)$, $R_p(\lambda,p)$, and $G(p)$ above do not depend on the particular solution $z(t)$ of~\eqref{eq:SistemaAuxiliar} with $z(0)>0$.

\begin{lemma} \label{lemma:indep}
We have the following:
\begin{enumerate}
\item Let $p>0$, $\eps > 0$ be sufficiently small and $0 < \theta \le K$. If
$$a,b \in \ ]\theta,K[ \quad \text{and} \quad  a-b<\eps,$$
then
\begin{equation} \label{eq:bdela}
b_\delta(p,t,a) - b_\delta(p,t,b) < \beta_S K_\theta p \eps .
\end{equation}
\item The numbers $R_p(\lambda,p)$ and $R_e(\lambda,p)$ and $G(p)$ are independent of the particular solution $z(t)$ with $z(0)>0$ of~\eqref{eq:SistemaAuxiliar}.
\end{enumerate}
\end{lemma}

We will also use the next technical lemma in the proof of our main theorem.
\begin{lemma} \label{lemma:tech}
If there is a positive constant $p>0$ such that $G(p)<0$ or $H(p)>0$ then there exists $T \ge 0$ such that either $W(p,t) \le 0$ for all $t \ge T$ or $W(p,t) > 0$ for all $t \ge T$. Additionally, if there are positive constants $p,\lambda>0$ such that $G(p)<0$ or $H(p)>0$, $R_p(\lambda,p)>1$ and $R_p^*(\lambda,p)>1$, then there exists $T \ge 0$ such that $W(p,t) \le 0$.
\end{lemma}

We say that the infectives go to extinction in in system~\eqref{eq:ProblemaPrincipal} if
    $$\lim_{t \to +\infty} I(t)=0$$
and we say that the infectives are strongly persistent in system~\eqref{eq:ProblemaPrincipal} if
    $$\liminf_{t \to +\infty} I(t)>0.$$

We now state our main theorem on the extinction and strong persistence of the infectives in system~\eqref{eq:ProblemaPrincipal}.

\begin{theorem}\label{teo:Main}
We have the following for system~\eqref{eq:ProblemaPrincipal}.
\begin{enumerate}
\item \label{teo:Extinction} If there are constants $\lambda>0$ and $p >0$ such that $R_e(\lambda,p)<1$, $R^*_e(\lambda,p)<1$ and $G(p)<0$
then the infectives $I$ go to extinction.
\item If there are constants $\lambda>0$ and $p >0$ such that $R_e(\lambda,p)<1$, $R^*_e(\lambda,p)<1$ and  $H(p)>0$
then the infectives $I$ go to extinction.
\item \label{teo:Permanence} If there are constants $\lambda>0$ and $p >0$ such that $R_p(\lambda,p)>1$, $R^*_p(\lambda,p)>1$ and $G(p)<0$ then the infectives $I$ are strongly persistent.
\item If there are constants $\lambda>0$ and $p >0$ such that $R_p(\lambda,p)>1$, $R^*_p(\lambda,p)>1$ and $H(p)>0$ then the infectives $I$ are strongly persistent.
\item \label{teo:Extinction-behavior} In the assumptions of~\ref{teo:Extinction}. any disease-free
solution $(S_1(t),0,0,R_1(t))$ is globally asymptotically stable.
\end{enumerate}
\end{theorem}

We also want to discuss the robustness of the conditions $R_e(\lambda,p)>0$, $R^*_e(\lambda,p)>0$, $R_p(\lambda,p)<0$, $R^*_p(\lambda,p)<0$, $H(p)>0$ and $G(p)<0$, i.e., roughly speaking if for sufficiently small perturbations of the parameters of our model in some admissible family of functions the conditions above are preserved. We will consider differentiable functions $\phi$.

Consider the family of systems
\begin{equation}\label{eq:ProblemaFamilia}
\begin{cases}
S'=\Lambda(t)-\beta_\tau(t)\,\phi_\tau(S,N,I)-\mu(t) S+\eta_\tau(t)R\\
E'=\beta_\tau(t)\,\phi_\tau(S,N,I)-(\mu(t)+\epsilon_\tau(t))E\\
I'=\epsilon_\tau(t)E -(\mu(t)+\gamma_\tau(t))I \\
R'=\gamma_\tau(t)I-(\mu(t)+\eta_\tau(t))R \\
N=S+E+I+R
\end{cases},
\end{equation}
where $\tau \in [-\zeta,\zeta]$ and we assume that, making $\tau=0$, we have $\phi_0=\phi$, $\beta_0=\beta$, $\eta_0=\eta$, $\epsilon_0=\epsilon$ and $\gamma_0=\gamma$ and that, for $\tau=0$ the parameters satisfy our assumptions (i.e. for $\tau=0$ we have our original system~\eqref{eq:ProblemaPrincipal}). We also assume that for each $\tau \in [-\zeta,\zeta]$ the parameter functions $\beta_\tau$, $\eta_\tau$, $\epsilon_\tau$ and $\gamma_\tau$ are continuous and bounded in $\R_0^+$, that $\phi_\tau$ is differentiable in $\Delta_{0,K}$ and that $\phi_\tau(x,n,0)=0$.

For $g:\R_0^+ \to \R$ denote by $\|\cdot\|_\infty$ the supremum norm (given by $\| g \|_\infty = \sup_{t \ge 0} |g(t)|$) and for $f:(\R_0^+)^3 \to \R$ denote by $\|\cdot\|_{\Delta_{0,K}}$ the $C^1$ norm of the restriction $f|_{\Delta_{0,K}}$:
$$\displaystyle \| f \|_{\Delta_{0,K}} = \max_{x \in \Delta_{0,K}} |f(x)| + \max_{x \in \Delta_{0,K}} \|d_x f\|.$$

Denote by $R_e^\tau(\lambda,p)$, $R_p^\tau(\lambda,p)$, $\left(R_e^*\right)^\tau(\lambda,p)$, $\left(R_p^*\right)^\tau(\lambda,p)$, $G^\tau_p(\lambda)$ and $H^\tau_p(\lambda)$, respectively the numbers~\eqref{eq:relambdap},~\eqref{eq:rplambdap},~\eqref{eq:r*elambdap},~\eqref{eq:r*plambdap}~\eqref{eq:Gp} and~\eqref{eq:Hp} with respect to the $\tau$ system in our family of models.

We have the following result on the robustness of conditions $R_e(\lambda,p)>0$, $R^*_e(\lambda,p)>0$, $R_p(\lambda,p)<0$, $R^*_p(\lambda,p)<0$, $H(p)>0$ and $G(p)<0$.

\begin{theorem}\label{teo:structural_stability}
Assume that $\|\beta_\tau-\beta\|_\infty$, $\|\eta_\tau-\eta\|_\infty$, $\|\epsilon_\tau-\epsilon\|_\infty$, $\|\gamma_\tau-\gamma\|_\infty $ and $\|\phi_\tau-\phi\|_{\Delta_{0,K}}$ converge to $0$ as $\tau \to 0$. Then there is $L>0$ such that, for all $\tau \in [-L,L]$, the numbers
$$\left|G^\tau(p)-G(p)\right|, \quad \left|H^\tau(p)-H(p)\right|, \quad \left|R_e^\tau(\lambda,p)-R_e(\lambda,p)\right|, $$
$$\left|R_p^\tau(\lambda,p)-R_p(\lambda,p)\right|, \quad \left|\left(R_e^*\right)^\tau(\lambda,p)-R_e^*(\lambda,p)\right| \quad  \text{and} \quad \left|\left(R_p^*\right)^\tau(\lambda,p)-R_p^*(\lambda,p)\right|$$
converge to $0$ as $\tau \to 0$.
\end{theorem}

The following is an immediate corollary.

\begin{corollary}
There is $L>0$ such that for all $\tau \in [-L,L]$ we have.
\begin{enumerate}
\item \label{teo:Ext} If there are constants $\lambda>0$ and $p >0$ such that $R_e(\lambda,p)<1$, $R^*_e(\lambda,p)<1$ and $G(p)<0$
then the infectives $I$ go to extinction in system~\eqref{eq:ProblemaFamilia}.
\item If there are constants $\lambda>0$ and $p >0$ such that $R_e(\lambda,p)<1$, $R^*_e(\lambda,p)<1$ and $H(p)>0$
then the infectives $I$ go to extinction in system~\eqref{eq:ProblemaFamilia}.
\item If there are constants $\lambda>0$ and $p >0$ such that $R_p(\lambda,p)>1$, $R^*_p(\lambda,p)>1$ and $G(p)<0$
then the infectives $I$ are strongly persistent in system~\eqref{eq:ProblemaFamilia}.
\item If there are constants $\lambda>0$ and $p >0$ such that $R_p(\lambda,p)>1$, $R^*_p(\lambda,p)>1$ and $H(p)>0$
then the infectives $I$ are strongly persistent in system~\eqref{eq:ProblemaFamilia}.
\item In the assumptions of~\ref{teo:Ext}. any disease-free
solution $(S_1(t),0,0,R_1(t))$ is globally asymptotically stable in system~\eqref{eq:ProblemaFamilia}.
\end{enumerate}
\end{corollary}

%-----------------------------------------------------------------------------%
\section{Examples} \label{section:E}
%-----------------------------------------------------------------------------%

\begin{example}[Autonomous case]
Letting $\Lambda(t)=\Lambda>0$, $\mu(t)=\mu>0$, $\eta(t)=\eta\ge0$, $\epsilon(t)=\epsilon\ge0$, $\gamma(t)=\gamma\ge0$ and $\beta(t)=\beta>0$ in~\eqref{eq:ProblemaPrincipal} and requiring that $\phi$ satisfies~H\ref{cond-CC1}) to H\ref{cond-CC2}) we obtain an autonomous SEIRS model verifying our assumptions. It is easy to see that $z(t)=\Lambda / \mu$ is a solution of~\eqref{eq:SistemaAuxiliar} with positive initial condition in this case. Letting
\begin{equation}\label{def:lim}
L_{\phi,\Lambda,\mu}=\lim_{\delta \to 0^+} \dfrac{\phi(\Lambda/\mu,\Lambda/\mu,\delta)}{\delta},
\end{equation}
we have
\begin{equation*}
  G(p)=\beta p L_{\phi,\Lambda,\mu} +\gamma - (1+1/p)\epsilon, \\[2mm]
\end{equation*}
\begin{equation*}
  H(p)=\gamma - \left(1+\frac{1}{p} \right) \epsilon, \\[2mm]
\end{equation*}
\begin{equation*}
  R_e(\lambda,p)= R_p(\lambda,p) = \Expo \left[ \left( \beta p L_{\phi,\Lambda,\mu} -\mu -\epsilon \right) \lambda\right],
\end{equation*}
and
\begin{equation*}
  R_e^*(\lambda,p)=R_p^*(\lambda,p)=\Expo\left[\left(\epsilon/p-\mu-\gamma\right)\lambda\right].
\end{equation*}
Define
\begin{equation}\label{eq:RA}
  R^{A}= \dfrac{\epsilon \beta \, L_{\phi,\Lambda,\mu}}{(\mu+\epsilon)(\mu+\gamma)}
\end{equation}
The following result is a consequence of Theorem~\ref{teo:Main} in the autonomous case.
\begin{corollary}\label{cor:aut}
We have the following for the autonomous system above.
\begin{enumerate}
\item \label{cor-aut-1} If $R^{A}<1$ then the infectives go to extinction;
\item \label{cor-aut-2} If $R^{A}>1$ then the infectives are strongly persistente;
\item \label{cor-aut-3} The disease free equilibrium $(\Lambda/\mu,0,0,0)$ is globally asymptotically stable.
\end{enumerate}
\end{corollary}

\begin{proof}
    Assuming that $R^A<1$ we have
    $$
    \dfrac{\epsilon \beta}{(\mu+\epsilon)(\mu+\gamma)} L_{\phi,\Lambda,\mu} < 1
    $$
  and thus for all $p>0$ such that
  \[
  \dfrac{\epsilon}{\mu+\gamma} < p < \dfrac{\mu+\epsilon}{\beta \displaystyle L_{\phi,\Lambda,\mu}},
  \]
  we have
  $$\frac{\epsilon}{p} < \mu+\gamma \quad \Leftrightarrow \quad \frac{\epsilon}{p}-\mu-\gamma<0 \quad \Leftrightarrow \quad R_e^*(\lambda,p)<1$$
  and also
  $$\beta p L_{\phi,\Lambda,\mu} < \mu+\epsilon \quad \Leftrightarrow \quad \beta p L_{\phi,\Lambda,\mu} - \mu - \epsilon <0 \quad \Leftrightarrow \quad R_e(\lambda,p)<1.$$
  Since
   $$G\left(\dfrac{\epsilon}{\mu+\gamma}\right)=\beta L_{\phi,\Lambda,\mu} \dfrac{\epsilon}{\mu+\gamma} +\gamma-\left(1+\dfrac{\mu+\gamma}{\epsilon}\right)\epsilon = (R^A-1)(\mu+\epsilon)<0$$ and $G$ is continuous we conclude that there is $p>0$ satisfying $R_e(\lambda,p)<1$, $R_e^*(\lambda,p)<1$ and $G(p)<0$. Thus, by~\ref{teo:Extinction}. in Theorem~\ref{teo:Main}, the infectives go to extinction and we obtain~\ref{cor-aut-1}..

  Assuming now that $R^A>1$ we have
    $$
    \dfrac{\epsilon \beta}{(\mu+\epsilon)(\mu+\gamma)} L_{\phi,\Lambda,\mu} > 1
    $$
  and thus, by the same reasoning, for all $p>0$ such that
  \[
  \dfrac{\epsilon}{\mu+\gamma} > p > \dfrac{\mu+\epsilon}{\beta \displaystyle L_{\phi,\Lambda,\mu}},
  \]
  we have $R_e^*(\lambda,p)>1$ and $R_e(\lambda,p)>1$.
  Since
   $$G\left(\dfrac{\mu+\epsilon}{\beta \displaystyle L_{\phi,\Lambda,\mu}}\right)=\beta L_{\phi,\Lambda,\mu} \dfrac{\mu+\epsilon}{\beta \displaystyle L_{\phi,\Lambda,\mu}} +\gamma-\left(1+\dfrac{\beta \displaystyle L_{\phi,\Lambda,\mu}}{\mu+\epsilon} \right)\epsilon =(\mu+\gamma)\left(1-R^A\right)<0$$
   and $G$ is continuous we conclude that there is $p>0$ satisfying $R_e(\lambda,p)<1$, $R_e^*(\lambda,p)<1$ and $G(p)<0$. Thus, by~\ref{teo:Permanence}. in Theorem~\ref{teo:Main}, the infectives are strongly persistent and we obtain~\ref{cor-aut-2}..

By \ref{teo:Extinction-behavior}. in~Theorem~\ref{teo:Main} we obtain immediatly~\ref{cor-aut-3}..
\end{proof}
Several particular forms for $\phi$ for particular SEIRS or SEIR model have been considered. For instance, in~\cite{Li-Zhou-CSF-2009}, for a SEIR autonomous model under different assumption than ours, an incidence of the form $\phi(S,N,I)=SI/(1+bN)$ with $b>0$ was considered. Also for a SEIR autonomous model~\cite{Driessche-Li-Muldowney-CAMQ-1999} a general incidence of the form $\phi(S,N,I)=g(I)S$ satisfying $g \in C^1$, $g(I)>0$, $g(0)=0$ and $\Lambda=\mu$ was considered. In~\cite{Buonomo-Lacitignola-RM-2008} an incidence of the form $\phi(S,N,I)=IS(1+\alpha I)$ with $\Lambda=\mu$ is considered. We can write our conditions for the previous incidence rates using Corollary~\ref{cor:aut}. For $\phi(S,N,I)=SI/(1+bN)$ we get the threshold $R^A=\epsilon\beta\Lambda/[(\mu+\epsilon)(\mu+\gamma)(\mu+b\Lambda)]$, for $\phi(S,N,I)=g(I)S$ with $g \in C^1$, $g(I)>0$, $g(0)=0$ and $\Lambda=\mu$ we obtain the threshold $R^A=\epsilon\beta g'(0)/[(\mu+\epsilon)(\mu+\gamma)]$ and for $\phi(S,N,I)=IS(1+\alpha I)$ we have the threshold $R^A=\epsilon\beta/[(\mu+\epsilon)(\mu+\gamma)]$.

\end{example}

\begin{example}[Asymptotically autonomous case]
In this section we are going to consider the asymptotically autonomous SEIRS model. That is, additionally to the assumptions on Theorem~\ref{teo:Main}, we are going to assume for system~\eqref{eq:ProblemaPrincipal} that the time-dependent parameters are asymptotically constant: $\mu(t) \to \mu$, $\eta(t) \to \eta$, $\epsilon(t)\to \epsilon$, $\gamma(t) \to \gamma$ and $\beta(t)\to\beta$ as $t \to +\infty$. Denoting by $F(t,x,y,z,w)$ the right hand side of~\eqref{eq:ProblemaPrincipal} and by $F_0(x,y,z,w)$ the right hand side of the limiting system, we also need to assume that
	$$\lim_{t \to +\infty} F(t,x,y,z,w) =F_0(x,y,z,w),$$
with uniform convergence on every compact set of $(\R_0^+)^4$ and we will also assume that $(x,y,z,w) \mapsto F(t,x,y,z,w)$ and $(x,y,z,w) \mapsto F_0(x,y,z,w)$ are locally Lipschitz functions.

There is a general setting that will allow us to study this case.
Namely, let $f:\R \times \R^n \to \R$ and $f_0:\R^n \to \R$ be continuous and locally Lipschitz in $\R^n$. Assume also that the non-autonomous system
\begin{equation}\label{eq:nonaut}
x'=f(t,x)
\end{equation}
is asymptotically autonomous with limit equation
\begin{equation}\label{eq:aut}
x'=f_0(x),
\end{equation}
that is, assume that $f(t,x) \to f_0(x)$ as $t \to +\infty$ with uniform convergence in every compact set of $\R^n$.
The following theorem is a particular case of a result established in~\cite{Markus-CTNO-1956} (for related results and applications see for example~\cite{Chavez-Thieme-MPD-1995, Mischaikow-Smith-Thieme-TAMS-1995}).
\begin{theorem}\label{Markus}
Let $\Phi(t,t_0,x_0)$ and $\phi(t,t_0,y_0)$ be solutions of~\eqref{eq:nonaut} and~\eqref{eq:aut} respectively. Suppose that $e \in \R^n$ is a locally stable equilibrium point of~\eqref{eq:aut} with attractive region
$$W(e)=\left\{y \in \R^n: \lim_{t \to +\infty} \phi(t,t_0,y) = e \right\}$$
and that $W_\Phi \cap W(e) \ne \emptyset$, where $W_\Phi$ denotes the omega limit of $\Phi(t,t_0,x_0)$. Then $\displaystyle \lim_{t \to +\infty} \Phi(t,t_0,x_0)=e$.
\end{theorem}

Since $(\R^+)^4$ is the attractive region for any solution of system~\eqref{eq:ProblemaPrincipal} with initial condition in $(\R^+)^4$ and the omega limit of every orbit of the asymptotically autonomous system with $I(t_0)>0$ is contained in $(\R^+)^4$, we can use Theorem~\ref{Markus} to obtain the following result.

\begin{corollary}
Let $R^A$ be the basic reproductive numbers of the limiting autonomous system, defined by~\eqref{eq:RA}. Then we have the following for the asymptotically autonomous systems above.
\begin{enumerate}
\item If $R^{A}<1$ then the infectives are extinct;
\item If $R^{A}>1$ then the infectives are strongly persistente.
\end{enumerate}
\end{corollary}
\end{example}

\begin{example}[Periodic model with constant $\Lambda$, $\mu$]
Next we assume that some model coefficients are periodic functions with the same period, namely we assume that there is $\omega>0$ such that, for all $t\ge 0$, we have $\eta(t)=\eta(t+\omega)$, $\epsilon(t)=\epsilon(t+\omega)$, $\gamma(t)=\gamma(t+\omega)$ and $\beta(t)=\beta(t+\omega)$. We also assume that $\mu$ and $\Lambda$ are constant functions and that $\phi$ satisfies~H\ref{cond-CC1}) to H\ref{cond-CC2}).

We have in his case
\[
R_e(\omega,p)<1 \ \Leftrightarrow \ \limsup_{t \to +\infty} \int_t^{t+\omega} \beta(s) L_{\phi,\Lambda,\mu} \ds \ \Leftrightarrow \ \left[ p \bar\beta L_{\phi,\Lambda,\mu} -\mu- \bar\epsilon\right] \omega <0\]
\[
R_e^*(\omega,p)<1 \ \Leftrightarrow \ \limsup_{t \to +\infty} \int_t^{t+\omega} \epsilon(s)/p-\mu-\gamma(s) \ds<0 \ \Leftrightarrow \ \left( \bar\epsilon /p- \mu-\bar\gamma \right)\omega <0,
\] \ \\[-1mm]
\[
G(p)=\max_{t \in [0,1]} \left[ \beta(t) p L_{\phi,\Lambda,\mu} +\gamma(t) - (1+1/p)\epsilon(t)\right], \\[2mm]
\]
\[
H(p)=\min_{t \in [0,1]} \left[ \gamma(t) - (1+1/p)\epsilon(t)\right], \\[2mm]
\]
Define
	$$R^{per}= \dfrac{\bar\epsilon \, \bar\beta \, L_{\phi,\Lambda,\mu}}{(\mu+\bar\epsilon)(\mu+\bar\gamma)}$$
where $\bar f=\frac{1}{\omega} \int_0^\omega f(s) \ds$ and $L_{\phi,\Lambda,\mu}$ is given by~\eqref{def:lim}.
The following result is a consequence of Theorem~\ref{teo:Main} in this case.
\begin{corollary}\label{cor:periodic}
We have for the periodic system with constant $\mu$ and $\Lambda$.
\begin{enumerate}
\item \label{cor-per-1} If $G\left( \bar\epsilon / (\mu+\bar\gamma) \right) <0$ or $H\left( (\mu+\bar\epsilon)/(\bar\beta L_{\phi,\Lambda,\mu}) \right) >0$ and $R^{per}<1$ then the infectives go to extinction;
\item \label{cor-per-2} If $G\left((\mu+\bar\epsilon)/(\bar\beta L_{\phi,\Lambda,\mu})\right) <0$ or $H\left(\bar\epsilon / (\mu+\bar\gamma)\right) >0$ and $R^{per}>1$ then the infectives are strongly persistent.
\end{enumerate}
\end{corollary}
\end{example}
\begin{proof}
By the same computations as in the proof of corollary~\ref{cor:aut} we conclude that $R_e^{per}<1$ if and only if there is
$$p \in I=\left(\dfrac{\bar\epsilon}{\mu+\bar\gamma}, \dfrac{\mu+\bar\epsilon}{\bar\beta \displaystyle L_{\phi,\Lambda,\mu}}\right)$$
such that $R_e(\lambda,p)<1$ and $R^*_e(\lambda,p)<1$ and that there is $\lambda>0$ such that $R_p^{per}(\lambda)>1$ if and only if there is $p\in I$ such that $R_p(\lambda,p)>1$ and $R^*_p(\lambda,p)>1$.

Moreover, by continuity of the functions $G$ and $H$, if
\[
  G\left(\dfrac{\bar\epsilon}{\mu+\bar\gamma}\right) <0 \quad \text{or} \quad H\left( \dfrac{\mu+\bar\epsilon}{\bar\beta L_{\phi,\Lambda,\mu}} \right) >0
\]
then there is $p \in I$ such that $G(p)<0$ or $H(p)>0$ and, by theorem~\ref{teo:Main}, we obtain~\ref{cor-per-1}..

By simmilar arguments we obtain~\ref{cor-per-2}..
\end{proof}
In~\cite{Rebelo-Margheri-Bacaër-2012}, a method to find threshold conditions in a general periodic epidemiological model relying in the spectral radius of some operator was obtained. Thought our conditions are not thresholds in the periodic case, they have the advantage that can be easily computed.

To illustrate the above corollary we consider the following family of periodic models
\begin{equation}\label{eq:ProblemaNakataKuniya}
\begin{cases}
S'=\mu-\beta(1+b\cos(2\pi t)) \, SI -\mu S+\eta R\\
E'=\beta(1+b\cos(2\pi t)) \, SI -(\mu+\epsilon(1+d\cos(2\pi t)))E\\
I'=\epsilon(1+d\cos(2\pi t)) E -(\mu+\gamma(1+k\cos(2\pi t)))I \\
R'=\gamma(1+k\cos(2\pi t)) I-(\mu +\eta)R \\
N=S+E+I+R
\end{cases}
\end{equation}
where $|b|<1$. In~\cite{Nakata-Kuniya-JMAA-2010} it was showed that for $\mu=2$, $\epsilon=1$, $\gamma=0.02$, $\eta=0.1$, $\beta=6.2$ and $b=0.6$ and $d=k=0$ the number $R^{per}$ is not a threshold. Our result is not applicable in this case since in this case $G\left(\epsilon / (\mu+\gamma) \right)= G(0.49505)=1.91089>0$. More generally it is easy to check that, for the system~\eqref{eq:ProblemaNakataKuniya}, letting $\beta$ and $b$ vary and $\mu=2$, $\epsilon=1$, $\gamma=0.02$, $\eta=0.1$ and $d=k=0$, we have that $R^{per}<1$ (respectively $R^{per}>1$) is equivalent to $\beta < 6.06$ (respectively $\beta > 6.06$),  $G(\epsilon/(\mu+\gamma))<0$ is equivalent to $\beta (1+|b|) <6.06$, $G\left( (\mu+\epsilon)/(\beta L_{\phi,\Lambda,\mu}) \right) <0$ is equivalent to $\beta>9|b|+6.06$ and $H\left( \epsilon/(\mu+\gamma)\right)>0$ and $H\left( (\mu+\epsilon)/(\beta L_{\phi,\Lambda,\mu}) \right)>0$ are impossible. In the first plot in figure~\ref{fig1} we plot the region of parameters $(b,\beta)$ where corollary~\ref{cor:periodic} is applicable and where we have extinction (purple) and permanence (blue).

 Using the parameters in~\cite{Nakata-Kuniya-JMAA-2010} but letting $\gamma$ and $k$ vary, we consider $\mu=2$, $\eta=0.1$, $\epsilon=1$, $\beta=6.06$ and $b=d=0$, we conclude that $G(\epsilon/(\mu+\gamma))<0$ is equivalent to $(2+\gamma) \left(3-\gamma|k|\right) > 6.06$, $G\left( (\mu+\epsilon)/(\beta L_{\phi,\Lambda,\mu}) \right) <0$ is equivalent to $\gamma(1+|k|)<0.02$, $H\left( \epsilon/(\mu+\gamma)\right)>0$ is impossible and $H\left( (\mu+\epsilon)/(\beta L_{\phi,\Lambda,\mu}) \right)>0$ is equivalent to $\gamma(1-|k|)>3.02$. Additionally $R^{per}<1$ is equivalent to $\gamma>0.02$ and $R^{per}>1$ is equivalent to $\gamma<0.02$. In the second plot in figure~\ref{fig1} we plot the region of parameters $(k,\gamma)$ where corollary~\ref{cor:periodic} is applicable and where we have extinction (purple) and permanence (blue).

 Finally, letting $\epsilon$ and $d$ vary and setting $\mu=2$, $\gamma=0.02$, $\eta=0.1$, $\beta=6.06$ and $b=k=0$, we conclude that $R^{per}<1$ is equivalent to $\epsilon<1$, $R^{per}>1$ is equivalent to $\epsilon>1$, $G(\epsilon/(\mu+\gamma))<0$ is equivalent to $2(\epsilon-1)+(2.02+\epsilon)|d|<0$, $G\left( (\mu+\epsilon)/(\beta L_{\phi,\Lambda,\mu}) \right) <0$ is equivalent to $|d|<1-(2.02+\epsilon)(2+\epsilon)/(\epsilon(8.06+\epsilon))$, $H\left( \epsilon/(\mu+\gamma)\right)>0$ is equivalent to $2.01 \epsilon (1+|b|)<0.02$ and $H\left( (\mu+\epsilon)/(\beta L_{\phi,\Lambda,\mu}) \right)>0$ is equivalent to $0.02(2+\epsilon)-(8.06+\epsilon)\epsilon(1+|b|)>0$. In the third plot in figure~\ref{fig1} we plot the region of parameters $(d,\epsilon)$ where corollary~\ref{cor:periodic} is applicable and where we have extinction (purple) and permanence (blue).

\begin{figure}[h!]
  \begin{minipage}[b][3.8cm]{.31\linewidth}
    \includegraphics[scale=0.38]{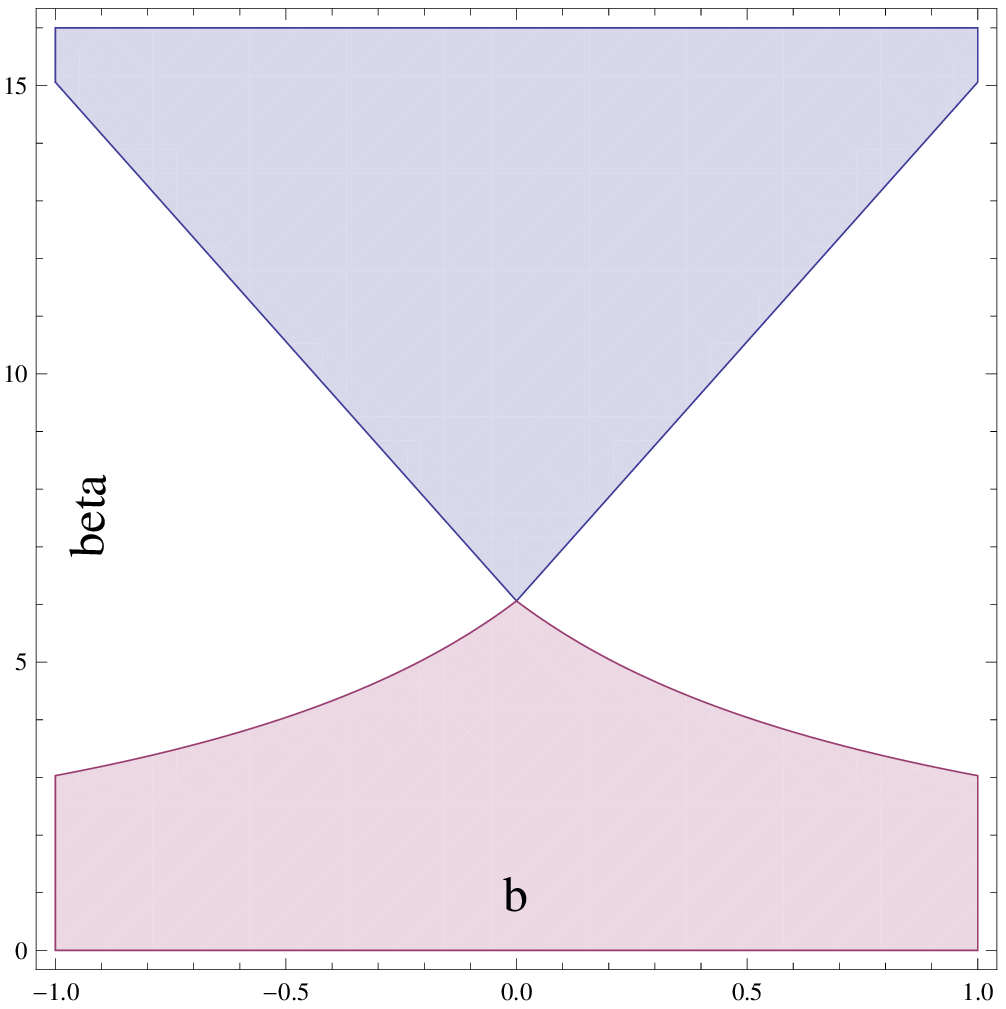}
  \end{minipage}
  \begin{minipage}[b][3.8cm]{.31\linewidth}
        \includegraphics[scale=0.38]{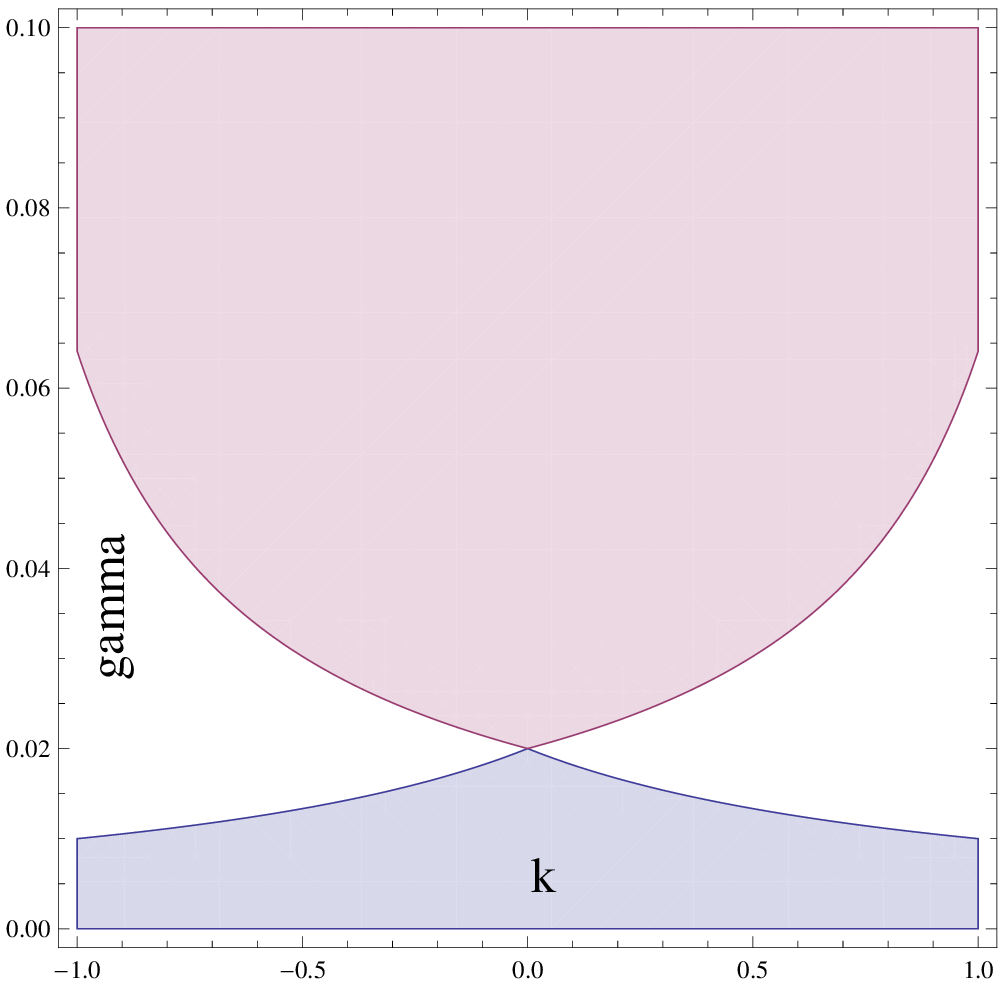}
  \end{minipage}
  \begin{minipage}[b][3.8cm]{.31\linewidth}
        \includegraphics[scale=0.38]{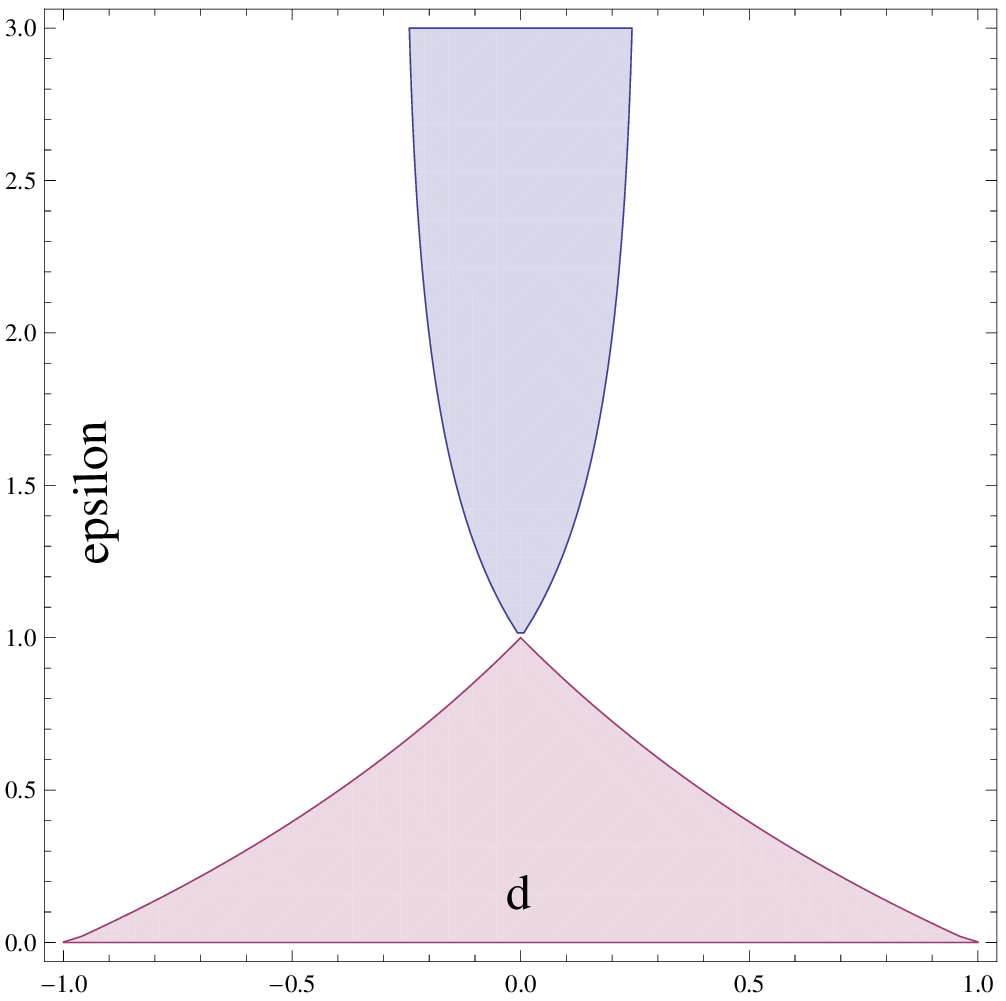}
  \end{minipage}
    \caption{first: $\mu=2$, $\epsilon=1$, $\gamma=0.02$, $\eta=0.1$ and $d=k=0$; second: $\mu=2$, $\epsilon=1$, $\eta=0.1$, $\beta=6.2$ and $b=d=0$; third: $\mu=2$, $\gamma=0.02$, $\eta=0.1$, $\beta=6.2$ and $b=k=0$}
      \label{fig1}
\end{figure}

\begin{example}[Michaelis-Menten contact rates]%\label{cor:Michaelis-Menten}
We consider the particular form for the incidence $\phi(S,N,I)=C(N)\frac{SI}{N}$. These rates are called Michaelis-Menten contact rates were considered for instance in~\cite{Zhang-Yingqi-Xu-AM-2013} and have as particular cases the standard incidence ($C(N)=1$) and the simple incidence ($C(N)=N$). We will assume that $\Lambda$ and $\mu$ are constant, that
\begin{equation}\label{eq:cond-MM1}
n \mapsto C(n)/n \quad \text{is non increasing}
\end{equation}
and that, for each $\theta>0$,
\begin{equation}\label{eq:cond-MM2}
\|C(n_1)-C(n_2)\| \le  K_\theta \|n_1-n_2\|.
\end{equation}

We have
\[
\begin{split}
  R_e(\lambda,p)<1 \quad
  & \Leftrightarrow \ \ \limsup_{t \to +\infty} \int_t^{t+\lambda} \beta(s)C(\Lambda/\mu)p-\mu-\epsilon(s) \ds<0 \\
  & \Leftarrow \ \ p C(\Lambda/\mu) \limsup_{t \to +\infty} \int_t^{t+\lambda} \beta(s) \ds - (\mu+\epsilon_\lambda^-)\lambda <0 \\
  & \Leftrightarrow \ \ \left[ p C(\Lambda/\mu) \beta_\lambda^+ -\mu-\epsilon_\lambda^- \right]\lambda <0
\end{split}
\]
\[
\begin{split}
  R_e^*(\lambda,p)<1 \quad
  & \Leftrightarrow \quad \limsup_{t \to +\infty} \int_t^{t+\lambda} \epsilon(s)/p-\mu-\gamma(s) \ds<0 \\
  & \Leftarrow \quad \left( \epsilon_\lambda^+ /p- \mu-\gamma_\lambda^- \right) \lambda <0, \\
\end{split}
\]
and analogously
$$R_p(\lambda,p)>1 \quad \Leftarrow \quad \left[ p C(\Lambda/\mu) \beta_\lambda^- -\mu-\epsilon_\lambda^+ \right]\lambda >0$$
and
$$R_p^*(\lambda,p)>1 \quad \Leftarrow \quad \left( \epsilon_\lambda^-  /p- \mu-\gamma_\lambda^+ \right) \lambda >0.$$
Define
$$R_e^{M}(\lambda)=\frac{\epsilon^+_\lambda \beta^+_\lambda  C(\Lambda/\mu)}{(\mu+\epsilon^-_\lambda)(\mu+\gamma^-_\lambda)} \quad \text{and} \quad R_p^{M}(\lambda)=\frac{\epsilon^-_\lambda \beta^-_\lambda  C(\Lambda/\mu)}{(\mu+\epsilon^+_\lambda)(\mu+\gamma^+_\lambda)}.$$
\begin{corollary}
We have the following for the Michaelis-Menten contact-rates with constant $\Lambda$ and $\mu$ and satisfying~\eqref{eq:cond-MM1} and~\eqref{eq:cond-MM2}.
\begin{enumerate}
\item If $G(\epsilon_\lambda^+/(\mu+\gamma_\lambda^-))<0$ or $H((\mu+\epsilon_\lambda^-)/(C(\Lambda/\mu)\beta_\lambda^+))>0$ and $R_e^{M}(\lambda)<1$ for some $\lambda>0$ then the infectives go to extinction;
\item If $G((\mu+\epsilon_\lambda^-)/(C(\Lambda/\mu)\beta_\lambda^+))<0$ or $H(\epsilon_\lambda^+/(\mu+\gamma_\lambda^-))>0$ and $R_p^{M}(\lambda)>1$ for some $\lambda>0$ then the infectives are strongly persistent.
\end{enumerate}
\end{corollary}

\begin{proof}
We begin by noting that there is $p>0$ such that $G(p)<0$ if and only if there is $p>0$ such that $pG(p)<0$. Since $pG(p)$ has two zeros, $a_0 \in \R^-$ and $a_1 \in \R^+$, and the coefficient of $p^2$ is positive, we conclude that there is $p>0$ such that $G(p)<0$ if and only if there is $p \in ]0,a_1[$ such that $G(p)<0$.

By the simmilar computations to the ones in the proof of corollary~\ref{cor:aut} we conclude that if there is $\lambda>0$ such that $R_e^{M}(\lambda)<1$ then there is
$$p \in I=\left(\dfrac{\epsilon_\lambda^+}{\mu+\gamma_\lambda^-},\dfrac{\mu+\epsilon_\lambda^-}{C(\Lambda/\mu)\beta_\lambda^+}\right)$$
such that $R_e(\lambda,p)<1$ and $R^*_e(\lambda,p)<1$. Thus, if $G(\epsilon_\lambda^+/(\mu+\gamma_\lambda^-))<0$ there is $p>0$ such that $]0,a_1[ \, \cap I \ne \emptyset$. Therefore if $G(\epsilon_\lambda^+/(\mu+\gamma_\lambda^-))<0$ there is there is $p>0$ such that $R_e(\lambda,p)<1$, $R_e^*(\lambda,p)<1$ and $G(p)<0$. Thus, by~Theorem~\ref{teo:Main}, the infectives go to extinction. On the other hand, since $H$ is continuous, if $H((\mu+\epsilon_\lambda^-)/(C(\Lambda/\mu)\beta_\lambda^+))>0$ there is $p \in I$ such that $R_e(\lambda,p)<1$, $R_e^*(\lambda,p)<1$ and $H(p)>0$. Therefore if $H((\mu+\epsilon_\lambda^-)/(C(\Lambda/\mu)\beta_\lambda^+))>0$ there is there is $p>0$ such that $R_e(\lambda,p)<1$, $R_e^*(\lambda,p)<1$ and $H(p)>0$. Thus, by~Theorem~\ref{teo:Main}, the infectives go to extinction and we obtain~\ref{cor-aut-1}..

By the simmilar computations we get~\ref{cor-aut-2}..
\end{proof}

In particular, setting $C(N)=N$ (mass-action incidence) we get
$$R_e^M(\lambda)=\frac{\epsilon^+_\lambda \beta^+_\lambda \Lambda/\mu}{(\mu+\epsilon^-_\lambda)(\mu+\gamma^-_\lambda)} \quad \text{and} \quad
R_p^M(\lambda)=\frac{\epsilon^-_\lambda \beta^-_\lambda \Lambda/\mu}{(\mu+\epsilon^+_\lambda)(\mu+\gamma^+_\lambda)}.$$
and setting $C(N)=1$ (standard incidence) we obtain
$$R_e^M(\lambda)=\frac{\epsilon^+_\lambda \beta^+_\lambda}{(\mu+\epsilon^-_\lambda)(\mu+\gamma^-_\lambda)} \quad \text{and} \quad
R_p^M(\lambda)=\frac{\epsilon^-_\lambda \beta^-_\lambda}{(\mu+\epsilon^+_\lambda)(\mu+\gamma^+_\lambda)}.$$
\end{example}

To illustrate the above corollary we consider the following family of nonperiodic models
\[
\begin{cases}
S'=\mu-\beta(1+b(1+\e^{-t})\cos(2\pi t)) \, SI -\mu S+\eta R\\
E'=\beta(1+b(1+\e^{-t})\cos(2\pi t)) \, SI -(\mu+\epsilon)E\\
I'=\epsilon E -(\mu+\gamma)I \\
R'=\gamma I-(\mu +\eta)R \\
N=S+E+I+R
\end{cases}
\]
It is easy to see that, in this case, $\beta^+_1=\beta^-_1=\beta$ and thus
$$R_e^M(\lambda)=R_p^M(\lambda)=\frac{\epsilon \beta}{(\mu+\epsilon)(\mu+\gamma)}$$
The following figures show situations where we have strong persistence and extinction for the above model with different values for $\beta$ and $b$ and $\mu=2$, $\epsilon=1$, $\gamma=0.02$ and $\eta=0.1$. For instance, for $\beta=10$ and $b=0.3$ we can see that $R_p^M(1)=1.65>1$ and $G(3/10)=-0.41<0$ and we conclude that we have strong persistence and for $\beta=5$ and $b=0.2$ we can see that $R_e^M(1)=0.82<1$ and $G(0.495)=-0.03<0$ and we conclude that we have extinction (see figure~\ref{fig2}).
\begin{figure}[h!]
  \begin{minipage}[b][3.8cm]{.49\linewidth}
    \includegraphics[scale=0.6]{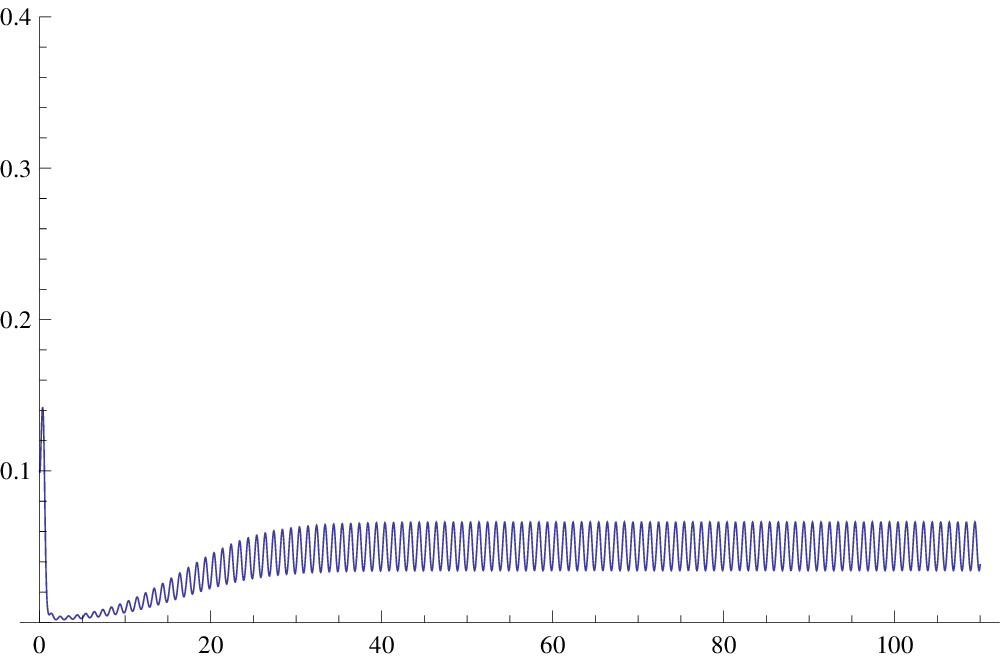}
  \end{minipage}
  \begin{minipage}[b][3.8cm]{.45\linewidth}
        \includegraphics[scale=0.6]{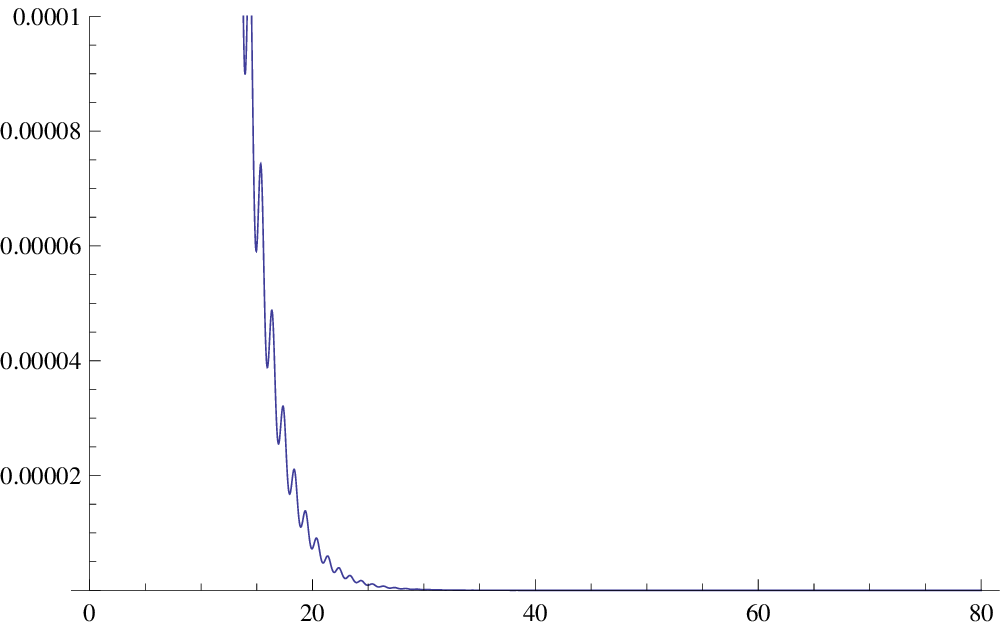}
  \end{minipage}
    \caption{left: $\beta=10$ and $b=0.3$; right: $\beta=5$ and $b=0.2$.}
      \label{fig2}
\end{figure}

%-----------------------------------------------------------------------------%
\section{Proofs} \label{section:P}
%-----------------------------------------------------------------------------%

%------------------------------------------------------------------------------
\subsection{Proof of Lemma~\ref{lemma:indep}}
%------------------------------------------------------------------------------
Assume that $p>0$, $\eps>0$ and $0\le\theta\le K$, $a,b \in ]\theta,K[$ and $a-b \le \eps$. We have, by~H\ref{cond-CC2}),
$$|\phi(a,a,\delta) - \phi(b,b,\delta)| \le K_\theta |a-b| \delta.$$
Therefore, if $a>b$ we have by~H\ref{cond-CC1})
\begin{equation}\label{eq:lem1-desig1}
\beta(t)\frac{\phi(a,a,\delta)}{\delta} - \beta(t)\frac{\phi(b,b,\delta)}{\delta} \le \beta(t)K_\theta |a-b|  = \beta(t)K_\theta (a-b) \le \beta_S K_\theta \eps
\end{equation}
and if $a<b$, again by~H\ref{cond-CC1}),
\begin{equation}\label{eq:lem1-desig2}
  \beta(t)\frac{\phi(a,a,\delta)}{\delta} - \beta(t)\frac{\phi(b,b,\delta)}{\delta} \le 0 \le \beta_S K_\theta \eps.
\end{equation}
By~\eqref{eq:lem1-desig1} and~\eqref{eq:lem1-desig2} we have
$$b_\delta(p,t,a)-b_\delta(p,t,b) \le \beta_S K_\theta p \eps$$
and we obtain~\eqref{eq:bdela}.

On the other side, again by~H\ref{cond-CC2}), assuming that $p>0$, $\eps>0$, $0\le\delta\le K$, $a,b \in ]\theta,K[$ and $|a-b| \le \eps$ we get
\[
\beta(t)\dfrac{\phi(a,a,\delta)}{\delta} - \beta_S K_\theta \eps \le \beta(t)\dfrac{\phi(b,b,\delta)}{\delta} \le \beta(t)\dfrac{\phi(a,a,\delta)}{\delta} + \beta_S K_\theta \eps,
\]
and thus
\begin{equation} \label{eq:supinftau}
b_\delta(p,t,a) - \beta_S K_\theta p \eps \le b_\delta(p,t,b) \le b_\delta(p,t,a) + \beta_S K_\theta p \eps.
\end{equation}

We will now show that $R_e(\lambda,p)$ and $R_e(\lambda,p)$ are independent of the particular solution $z(t)$ of~\eqref{eq:SistemaAuxiliar} with $z(0)>0$. In fact, letting $z_1$ be some solution of~\eqref{eq:SistemaAuxiliar} with $z_1(0)>0$, by~\ref{cond-4-aux}) in Proposition~\ref{eq:subsectionAS}, we can choose $\theta_1 > 0$ such that $z(t),z_1(t) \ge \theta_1$ for all $t \ge T$. On the other hand, by~\ref{cond-2-aux}) in Proposition~\ref{eq:subsectionAS}, given $\eps>0$ there is a $T_\eps >0$ such that $|z(t)- z_1(t)| < \eps$ for every $t \ge T_\eps$. Letting $a=z(t)$ and $b=z_1(t)$ and computing the integral from $t$ to $t+\lambda$ in~\eqref{eq:supinftau} we get
\[
\left| \int_t^{t+\lambda} \lim_{\delta \to 0^+} b_\delta(p,s,z_1(s)) \ds - \int_t^{t+\lambda} \lim_{\delta \to 0^+} b_\delta(p,s,z(s)) \ds \right| \le \lambda \beta_S K_{\theta_1} p \, \eps,
\]
for every $t \ge T_\eps$. We conclude that, for every $\eps>0$,
\[
\begin{split}
& \limsup_{t \to +\infty} \int_t^{t+\lambda} \lim_{\delta \to 0^+} b_\delta(p,s,z_1(s)) \ds - \lambda \beta_S K_{\theta_1} p \, \eps \\
& \le \limsup_{t \to +\infty} \int_t^{t+\lambda} \lim_{\delta \to 0^+} b_\delta(p,s,z(s)) \ds \\
& \le \limsup_{t \to +\infty} \int_t^{t+\lambda} \lim_{\delta \to 0^+} b_\delta(p,s,z_1(s)) \ds + \lambda \beta_S K_{\theta_1} p \, \eps,
\end{split}
\]
and thus $R_e(\lambda,p)$ is independent of the chosen solution. Taking $\liminf$ instead of $\limsup$, the same reasoning shows that $R_p(\lambda,p)$ is also independent of the particular solution. Similar computations imply that $G(p)$ is also independent of the particular chosen solution. This proves the lemma.
%------------------------------------------------------------------------------
\subsection{Proof of Lemma~\ref{lemma:tech}.}
%------------------------------------------------------------------------------

Lets assume first that $G(p)<0$ and let $(S(t),E(t),I(t),R(t))$ be some solution of~\eqref{eq:ProblemaPrincipal} with $S(T_0), E(T_0), I(T_0), R(T_0)>0$ for some $T_0>0$. Then there is $T_1>0$ such that $\displaystyle g_\delta(p,t,N(t))>0$ for all $t \ge T_1$ (note that $N(t)$ is a solution of~\eqref{eq:SistemaAuxiliar}). By contradiction, assume also that there is no $T_2 \ge T_1$ such that $W(p,t) \le 0$ or $W(p,t) > 0$ for all $t \ge T_2$. Therefore there is $s \ge T_1$ such that
$$W(p,s)=0 \quad \Leftrightarrow \quad pE(s)=I(s)$$
and
$$\frac{dW}{dt}(p,s) >0.$$

Since $s \ge T_1$ we have $\displaystyle \lim_{\delta \to 0^+} g_\delta(p,s,N(s))<0$. By~H\ref{cond-CC1}),~H\ref{cond-CC3}) and~\eqref{eq:glambda} we obtain
{\small
\[
\begin{split}
0
& < \frac{dW}{dt}(p,s) \\
& = \frac{d}{dt} [pE(t)-I(t)]|_{t=s} \\
& = p E'(s) - I'(s) \\
& = p \left[\beta(s) \, \phi(S(s),N(s),I(s)) - (\mu(s)+\epsilon(s))E(s) \right]- \epsilon(s)E(s)+(\mu(s)+\gamma(s))I(s)\\
& = \left[p\beta(s) \, \frac{\phi(S(s),N(s),I(s))}{I(s)} + \mu(s)+\gamma(s) \right]I(s) - [p (\mu(s)+\epsilon(s)) + \epsilon(s)] E(s) \\
& \le \left[p\beta(s)\lim_{\delta \to 0^+} \frac{\phi(S(s),N(s),\delta)}{\delta} + \mu(s)+\gamma(s) \right]I(s) - \left[\mu(s)+\epsilon(s)+\frac{\epsilon(s)}{p}\right] pE(s) \\
& =\left[p\beta(s)\lim_{\delta \to 0^+} \frac{\phi(S(s),N(s),\delta)}{\delta} + \gamma(s) - \epsilon(s)\left( 1 + \frac{1}{p}\right)\right]I(s) \\
& \le\left[p\beta(s)\lim_{\delta \to 0^+} \frac{\phi(N(s),N(s),\delta)}{\delta} + \gamma(s) - \epsilon(s)\left( 1 + \frac{1}{p}\right)\right]I(s) \\
& = \lim_{\delta \to 0^+} g_\delta(p,s,N(s)) I(s) \le 0
\end{split}
\]
}
witch contradicts the assumption. Thus there is $T_2 \ge T_1$ such that $W(p,t) \le 0$ or $W(p,t) > 0$ for all $t \ge T_2$.

Assume now that $H(p)\ge 0$ and let $(S(t),E(t),I(t),R(t))$ be some solution of~\eqref{eq:ProblemaPrincipal} with $S(T_0), E(T_0), I(T_0), R(T_0)>0$ for some $T_0>0$. Then there is $T_3>0$ such that $ h(p,t)>0$ for all $t \ge T_3$. By contradiction, assume also that there is no $T_4 \ge T_3$ such that $W(p,t) \le 0$ or $W(p,t) > 0$ for all $t \ge T_4$. Therefore there is $s \ge T_3$ such that
$$W(p,s)=0 \quad \Leftrightarrow \quad pE(s)=I(s)$$
and
$$\frac{dW}{dt}(p,s) <0.$$
Since $s \ge T_3$ we have $\displaystyle h(p,s)>0$. By~H\ref{cond-CC1}),~H\ref{cond-CC3}) and~\eqref{eq:glambda} we obtain
{\small
\[
\begin{split}
0
& > \frac{dW}{dt}(p,s) \\
& = \frac{d}{dt} [pE(t)-I(t)]|_{t=s} \\
& = p E'(s) - I'(s) \\
& = p \left[\beta(s) \, \phi(S(s),N(s),I(s)) - (\mu(s)+\epsilon(s))E(s) \right]- \epsilon(s)E(s)+(\mu(s)+\gamma(s))I(s)\\
& \ge \left[ \mu(s)+\gamma(s) \right]I(s) - \left[\mu(s)+\epsilon(s)+\frac{\epsilon(s)}{p}\right] pE(s) \\
& =\left[\gamma(s) - \epsilon(s)\left( 1 + \frac{1}{p}\right)\right]I(s) \\
& =  h(p,s) I(s) \ge 0
\end{split}
\]
}
witch is a contradiction. Thus there is $T_4 \ge T_3$ such that $W(p,t) \le 0$ or $W(p,t) > 0$ for all $t \ge T_4$.
Assuming that $G(p)<0$ or $H(p)>0$, $R_p(\lambda,p)>1$ and $R_p^*(\lambda,p)>1$ for some $p,\lambda>0$, by the previous arguments, we have $W(p,t)>0$ for all $t \ge T_2$ or $W(p,t) \le 0$ for all $t \ge T_2$. Suppose by contradiction that $W(p,t)>0$ for all $t \ge T_2$. We have $E(t)>I(t)/p$ for all $t \ge T_2$. Then, by the third equations in~\eqref{eq:ProblemaPrincipal} we have
$$\frac{d}{dt}I(t) > \epsilon(t) \frac{1}{p} I(t) - (\mu(t)+\gamma(t)) I(t) = [\epsilon(t) \frac{1}{p} - \mu(t)-\gamma(t) ] I(t)$$
and thus, for all $t \ge T_2$, we have
$$I(t)>I(T_2) \e^{\int_{T_2}^t \epsilon(r) \frac{1}{p} - \mu(r)-\gamma(r) \dr}.$$
Since $R^*_p(\lambda,p)>1$, by~\eqref{eq:r*plambdap} we conclude that there is $\eta>0$ and $T>0$ such that, for all $t \ge T$, we have
$$\int_t^{t+\lambda} \epsilon(r) \frac{1}{p} - \mu(r)-\gamma(r) \dr > \eta.$$
Thus, for all $t >  \max\{T_2,T\}$, we obtain $I(t)> I(T_2) \e^{\left(\frac{t-T_2}{\lambda} -1 \right)\eta}$. Thus $I(t) \to +\infty$ and this contradicts the fact that $I(t)$ must be bounded. Then we must have $W(p,t)\le 0$ and lemma is proved.

%------------------------------------------------------------------------------
\subsection{Proof of Theorem~\ref{teo:Main}.}
%------------------------------------------------------------------------------
Assume that there are constants $\lambda>0$ and $p >0$ such that
$R_e(\lambda,p)<1$, $R^*_e(\lambda,p)<1$ and $G(p)<0$ and let $(S(t),E(t),I(t),R(t))$ be some solution of~\eqref{eq:ProblemaPrincipal} with $S(T_0), E(T_0), I(T_0), R(T_0)>0$ for some $T_0>0$. By contradiction, assume that $\displaystyle \liminf_{t\to +\infty} I(t)>0$ and thus that there are $T\ge T_0$ and $\eps_0>0$ and such that $I(t)>\eps_0$ for all $t>T$.

Since $R_e(\lambda,p)<1$, by~\eqref{eq:relambdap} we conclude that there is $T_1 \ge T$ such that
\[
 \int_t^{t+\lambda} \lim_{\delta \to 0^+} b_{\delta}(p,s,N(s)) \ds < -\eta < 0,
\]
for all $t \ge T_1$.

By \ref{cond-3-bs}) in Proposition~\ref{subsection:BS}, we may assume that $(S(t),N(t),I(t)) \in \Delta_{0,k}$ for $t \ge T_1$.

By Lemma~\ref{lemma:tech} we have $W(p,t) > 0$ for all $t \ge T_1$ or
$W(p,t) \le 0$ for all $t \ge T_1$. Assume first that $W(p,t) > 0$ for all $t \ge T_1$.
Since $I(T_0)>0$, by~\ref{cond-2-bs}) in Proposition~\ref{subsection:BS} we have that $I(t) >0$ for all $t \ge T_0$ and, by the second equation in~\eqref{eq:ProblemaPrincipal},~H\ref{cond-CC1}),~H\ref{cond-CC3}) and~\eqref{eq:bSlambda}, there is $T_2 \ge T_1$ such that
\begin{equation}\label{eq:integrar1}
\begin{split}
E'(t)
& = \beta(t) \phi(S(t),N(t),I(t))-(\mu(t)+\epsilon(t))E(t) \\
& = \beta(t) \frac{\phi(S(t),N(t),I(t))}{I(t)} I(t)-(\mu(t)+\epsilon(t))E(t) \\
& < \beta(t) \frac{\phi(N(t),N(t),I(t))}{I(t)} pE(t)-(\mu(t)+\epsilon(t))E(t) \\
& \le \beta(t) \lim_{\delta \to 0^+} \frac{\phi(N(t),N(t),\delta)}{\delta} pE(t)-(\mu(t)+\epsilon(t))E(t) \\
& = \lim_{\delta \to 0^+} b_\delta(p,t,N(t))E(t)
\end{split}
\end{equation}
for all $t \ge T_2$ and $0 < \delta \le \delta_1$.
Thus, integrating~\eqref{eq:integrar1} we obtain
\[
\begin{split}
E(t)
& \le E(T_2) \Expo\left[\int_{T_2}^t \displaystyle \lim_{\delta \to 0^+} b_\delta(p,s,N(s)) \ds\right]  \\
& = E(T_2) \Expo\left[\int_{T_2}^{T_2+\lambda \lfloor \frac{t-T_2}{\lambda} \rfloor} \displaystyle \lim_{\delta \to 0^+}b_\delta(p,s,N(s)) \ds \, + \right. \\
& \hspace{4cm} \left.+\int_{T_2+\lambda \lfloor \frac{t-T_2}{\lambda} \rfloor}^t \displaystyle \lim_{\delta \to 0^+} b_\delta(p,s,N(s)) \ds\right] \\
& \le E(T_2) \Expo\left[ \int_{T_2}^{T_2+\lambda \lfloor \frac{t-T_2}{\lambda} \rfloor} \displaystyle \lim_{\delta \to 0^+} b_\delta(p,s,N(s)) \ds \, + \right. \\
& \hspace{4cm} \left. + \int_{T_2+\lambda \lfloor \frac{t-T_2}{\lambda} \rfloor}^t \displaystyle \beta(s) \lim_{\delta \to 0^+} \frac{\phi(N(s),N(s),\delta)}{\delta}p \ds\right] \\
& < E(T_2) \Expo\left[-\eta \lfloor \frac{t-T_2}{\lambda}\rfloor +\beta_SMp\lambda\right],
\end{split}
\]
for all $t \ge T_2$. We conclude that $0 \le \displaystyle \limsup_{t \to +\infty} I(t) \le p \displaystyle \limsup_{t \to +\infty} E(t) = 0$ assuming that $W(p,t)>0$ for all $t \ge T_1$.

Assume now that $W(p,t) \le 0$ for all $t \ge T_1$.
By the third equation in~\eqref{eq:ProblemaPrincipal} we have
\begin{equation}\label{eq:majora-I}
I'(t)\le\epsilon(t)I(t)/p -(\mu(t)+\gamma(t))I(t)
=(\epsilon(t)/p-\mu(t)-\gamma(t))I(t)
\end{equation}
for all $t \ge T_1$. Since $R_e^*(\lambda,p)<1$, by~\eqref{eq:r*elambdap} we conclude that there are constants $\eta_2>0$ and $T_3 \ge T_1$ such that
\begin{equation}\label{eq:r*<0}
\int_t^{t+\lambda}  \epsilon(s)/p-\mu(s)-\gamma(s) \ds < -\eta_2 < 0,
\end{equation}
for all $t \ge T_3$. Thus,  by~\eqref{eq:majora-I} and~\eqref{eq:r*<0}, we have
$$I(t) \le I(T_3) \e^{\int_{T_3}^t \epsilon(s)/p-\mu(s)-\gamma(s) \ds} \le I(T_3) \e^{-\eta_2\lfloor\frac{t-T_3}{\lambda}\rfloor +\frac{\lambda\varepsilon_S}{p}},$$
for all $t \ge T_3$. We conclude that $I(t) \to 0$, assuming that $W(p,t) \le 0$ for all $t \ge T_1$. Therefore we obtain~\ref{teo:Extinction}. in the theorem.

Assume now that there are constants $\lambda>0$, $p >0$ such that
$R_p(\lambda,p)>1$, $R^*_p(\lambda,p)>1$ and $G(p)<0$ for all $t \ge T$ and let $(S(t),E(t),I(t),R(t))$ be some fixed solution of~\eqref{eq:ProblemaPrincipal} with $S(T_0), E(T_0), I(T_0), R(T_0)>0$ for some $T_0>0$.

Since $R_p(\lambda,p)>1$, by~\eqref{eq:rplambdap} and~H\ref{cond-CC5}) we conclude that there are constants $0 < \delta_2 \le K$, $\eta>0$ and $T_4>0$ such that
\begin{equation}\label{eq:persist-condition}
\int_t^{t+\lambda} \beta(s) \frac{\phi(N(s),N(s),\delta)}{\delta} \, p - \mu(s) - \epsilon(s) \ds > \eta > 0,
\end{equation}
for all $t \ge T_4$ and $0< \delta \le \delta_2$ and that $g_\delta(p,t,N(t))<0$ for all $t \ge T_5$ and $0< \delta \le \delta_2$. By Proposition~\ref{subsection:BS}, we may also assume that $(S(t),N(t),I(t)) \in \Delta_{0,K}$ for all $t \ge T_4$.

By~\eqref{eq:d-Lambda-} we can choose $\eps_1>0$, $0< \eps_2 <\delta_2$, $\eps_3>0$  and $0<\eta_1 < \eta$ such that, for all $t \ge T_4$, we have
\begin{equation}\label{cond-int-1-contrad}
\int_t^{t+\lambda} \beta(s)M \eps_2 - (\mu(s) + \epsilon(s)) \eps_1 \ds < -\eta_1
\end{equation}
\begin{equation}\label{cond-int-2-contrad}
\int_t^{t+\lambda} \gamma(s) \eps_2 - (\mu(s)+\eta(s))\eps_3 \ds < -\eta_1
\end{equation}
\begin{equation}\label{cond-4-contrad}
\theta_1=\frac{m_1}{2}-\eps_1-[1+\beta_S M\lambda+\gamma_S\lambda]\eps_2-\eps_3>0
\end{equation}
and
\begin{equation}\label{cond-int-3-contrad}
\kappa = K_{\theta_1}[\eps_1+[1+\beta_S M\lambda+\gamma_S\lambda]\eps_2+\eps_3] < \frac{\eta}{2p\beta_S\lambda}
\end{equation}
where $M$ is given by~\eqref{cond-CC4}.

We will show that
\begin{equation}\label{eq:limsupI>eps2}
  \limsup_{t \to +\infty} I(t) > \eps_2.
\end{equation}
Assume by contradiction that it is not true. Then there exists $T_5>T_4$ such that, for all $t \ge T_5$, we have
\begin{equation}\label{eq:contradictioneps2}
I(t) \le \eps_2.
\end{equation}

Suppose that $E(t) \ge \eps_1$ for all $t \ge T_5$. Then, by the second equation in~\eqref{eq:ProblemaPrincipal},~\eqref{cond-CC4}, H\ref{cond-CC3}) and \eqref{cond-int-1-contrad}, we have for all $t \ge T_5$
\[
\begin{split}
E(t)
& = E(T_5)+\int_{T_5}^t \beta(s) \, \phi(S(s),N(s),I(s)) - (\mu(s) + \epsilon(s))E(s) \ds \\
& = E(T_5)+\int_{T_5}^t \beta(s) \, \frac{\phi(S(s),N(s),I(s))}{I(s)} I(s) - (\mu(s) + \epsilon(s))E(s) \ds \\
& \le E(T_5)+\int_{T_5}^t \beta(s) M \eps_2  - (\mu(s) + \epsilon(s))\eps_1 \ds\\
& = E(T_5)+\int_{T_5}^{T_5+\lfloor \frac{t-T_5}{\lambda} \rfloor \lambda} \beta(s) M \eps_2  - (\mu(s) + \epsilon(s))\eps_1 \ds \\
& \quad \quad + \int_{T_5+\lfloor \frac{t-T_5}{\lambda} \rfloor \lambda}^t \beta(s) M \eps_2 - (\mu(s) + \epsilon(s))\eps_1 \ds\\
& < E(T_5) - \eta_1 \lfloor \frac{t-T_5}{\lambda} \rfloor + \beta_S M \eps_2 \lambda
\end{split}
\]
and thus $E(t) \to -\infty$ witch contradicts~\ref{cond-2-bs}) in Proposition~\ref{subsection:BS}. We conclude that there exists $T_6 \ge T_5$ such that $E(T_6)<\eps_1$. Suppose that there exists a $T_7 > T_6$ such that $E(T_7)> \eps_1 + \beta_S M \eps_2\lambda$. Then we conclude that there must exist $T_8 \in ]T_6,T_7[$ such that $E(T_8)=\eps_1$ and $E(t)>\eps_1$ for all $t \in ]T_8,T_7]$. Let $n \in \N_0$ be such that $T_7 \in [T_8+n \lambda, T_8 + (n+1)\lambda]$. Then, by the second equation in~\eqref{eq:ProblemaPrincipal},~\eqref{cond-CC4},~\eqref{eq:contradictioneps2} and~\eqref{cond-int-1-contrad} we have
\[
\begin{split}
\quad \eps_1 + \beta_SM \eps_2 \lambda
& < E(T_7) \\
& = E(T_8) + \int_{T_8}^{T_7}
\beta(s) \phi(S(s),N(s),I(s)) - (\mu(s)+\epsilon(s))E(s) \ds \\
& < E(T_8) + \int_{T_8}^{T_7}
\beta(s) M \eps_2 - (\mu(s)+ \epsilon(s))\eps_1 \ds \\
& \le \eps_1 -\eta_1 n + \int_{T_8+n\lambda}^{T_7} \beta_S M \eps_2 \ds \\
& \le \eps_1 + \beta_S M \eps_2 \lambda
\end{split}
\]
and this is a contradiction. We conclude that, for all $t \ge T_7$ we have
\begin{equation}\label{eq:boundE}
E(t) \le \eps_1 + \beta_S M \eps_2 \lambda.
\end{equation}

Suppose that $R(t) \ge \eps_3$ for all $t \ge T_9$. Then,
by the fourth equation in~\eqref{eq:ProblemaPrincipal},~\eqref{eq:contradictioneps2} and \eqref{cond-int-2-contrad}, we have for all $t \ge T_9$
\[
\begin{split}
R(t)
& = R(T_9)+\int_{T_9}^t \gamma(s)I(s) - (\mu(s) + \eta(s))R(s) \ds \\
& \le R(T_9)+\int_{T_9}^t \gamma(s)\eps_2 - (\mu(s) + \eta(s))\eps_3 \ds \\
& = R(T_9)+\int_{T_9}^{T_9+\lambda \lfloor \frac{t-T_9}{\lambda} \rfloor} \gamma(s)\eps_2 - (\mu(s) + \eta(s))\eps_3 \ds \\
& \quad \quad +\int_{T_9+\lambda \lfloor \frac{t-T_9}{\lambda} \rfloor}^t \gamma(s)\eps_2 - (\mu(s) + \eta(s))\eps_3 \ds \\
& < R(T_9) - \eta_1 \lfloor \frac{t-T_9}{\lambda} \rfloor + \gamma_S \eps_2 \lambda
\end{split}
\]
and thus $R(t) \to -\infty$ witch contradicts~\ref{cond-2-bs}) in Proposition~\ref{subsection:BS}. We conclude that there exists $T_{10} \ge T_9$ such that $R(T_{10})<\eps_3$. Suppose that there exists a $T_{11} \ge T_{10}$ such that $R(T_{11})> \eps_3 + \gamma_S \eps_2 \lambda$. Then we conclude that there must exist $T_{12} \in ]T_{10},T_{11}[$ such that $R(T_{12})=\eps_3$ and $R(t)>\eps_3$ for all $t \in ]T_{12},T_{11}]$. Let $n \in \N_0$ be such that $T_{11} \in [T_{12}+n \lambda, T_{12} + (n+1)\lambda]$. Then, by the fourth equation in~\eqref{eq:ProblemaPrincipal},~\eqref{eq:contradictioneps2} and \eqref{cond-int-2-contrad} we have
\[
\begin{split}
\quad \eps_3 + \gamma_S \eps_2 \lambda
& < R(T_{11}) \\
& = R(T_{12}) + \int_{T_{12}}^{T_{11}} \gamma(s)I(s)-(\mu(s)+\eta(s))R(s)\ds \\
& < R(T_{12})+\int_{T_{12}}^{T_{11}}\gamma(s)\eps_2-(\mu(s)+\eta(s))\eps_3 \ds\\
& < \eps_3 -\eta_1 n + \int_{T_{12}+n\lambda}^{T_{11}} \gamma_S \eps_2 \ds \\
& \le \eps_3 + \gamma_S \eps_2\lambda
\end{split}
\]
and this is a contradiction. We conclude that, for all $t \ge T_{10}$ we have
\begin{equation}\label{eq:boundR}
R(t) \le \eps_3 + \gamma_S \eps_2\lambda.
\end{equation}

By Lemma~\ref{lemma:tech} there exists $T_{13} \ge T_{10}$ such that $p E(t) \le I(t)$, for all $t \ge T_{13}$. According to the second equation in~\eqref{eq:ProblemaPrincipal} and~H\ref{cond-CC3}) and recalling that by~\eqref{eq:contradictioneps2} and the assumptions we have $I(t) \le \eps_2 < \delta_2$, for all $t \ge T_{13}$ we get,
\begin{equation}\label{eq:majE-inf}
\begin{split}
E'(t)
& = \beta(t) \phi(S(t),N(t),I(t))-(\mu(t)+\epsilon(t))E(t) \\
& = \beta(t) \frac{\phi(S(t),N(t),I(t))}{I(t)}I(t)-(\mu(t)+\epsilon(t))E(t) \\
& \ge \beta(t) \frac{\phi(S(t),N(t),\delta_2)}{\delta_2}I(t)
-(\mu(t)+\epsilon(t))E(t)
\end{split}
\end{equation}

By~\eqref{eq:contradictioneps2},~\eqref{eq:boundE} and~\eqref{eq:boundR}, we have, for all $t \ge T_{13}$,
\begin{equation}\label{eq:s-to-z-1}
\begin{split}
N(t)-S(t) & =  E(t)+I(t)+R(t) \\
& \le \eps_1+\beta_S M \eps_2 \lambda +\eps_2+\eps_3+\gamma_S \eps_2 \lambda\\
& = \eps_1+[1+\beta_S M\lambda+\gamma_S\lambda]\eps_2+\eps_3.
\end{split}
\end{equation}
Un the other side, by~\ref{cond-4-aux}) in Proposition~\ref{eq:subsectionAS}, there is $T_{14}>T_{13}$ such that, for all $t \ge T_{14}$, we have $N(t)\ge m_1/2$. Therefore, for all $t \ge T_{14}$, we have by~\eqref{eq:s-to-z-1} and~\eqref{cond-4-contrad}
\[
\begin{split}
S(t)
& \ge N(t)-\eps_1-[1+\beta_S M\lambda+\gamma_S\lambda]\eps_2-\eps_3\\
& \ge \frac{m_1}{2}-\eps_1-[1+\beta_S M\lambda+\gamma_S\lambda]\eps_2-\eps_3 \\
& = \theta_1 >0.
\end{split}
\]
Thus, by~H\ref{cond-CC2}),~\eqref{eq:s-to-z-1} and~\eqref{cond-int-3-contrad} we have
\[
\begin{split}
  |\phi(S(t),N(t),\delta_2)-\phi(N(t),N(t),\delta_2)|
  & \le K_{\theta_1} |S(t)-N(t)|\delta_2 \\
  & \le K_{\theta_1} [\eps_1+[1+\beta_S M\lambda+\gamma_S\lambda]\eps_2+\eps_3] \delta_2\\
  & = \kappa \delta_2.
\end{split}
\]
Therefore, by~\eqref{eq:majE-inf},~\eqref{eq:s-to-z-1},~\eqref{cond-int-3-contrad}, H\ref{cond-CC2}) and since $p E(t) \le I(t)$, we obtain, for all $t \ge T_{14}$,
\begin{equation}\label{eq:majE-inf2}
\begin{split}
E'(t)
& \ge \beta(t) \ \frac{\phi(N(t),N(t),\delta_2)-\kappa \delta_2}{\delta_2} \ I(t) -(\mu(t)+\epsilon(t))E(t)\\
& = \left[\beta(t) \frac{\phi(N(t),N(t),\delta_2)}{\delta_2}-\beta(t)\kappa \right]I(t)
-(\mu(t)+\epsilon(t))E(t) \\
& \ge \left[ \beta(t) \frac{\phi(N(t),N(t),\delta_2)}{\delta_2}p-\beta(t)\kappa p
-\mu(t)-\epsilon(t)\right] E(t).
\end{split}
\end{equation}

Therefore, integrating~\eqref{eq:majE-inf2} an using~\eqref{eq:persist-condition} and~\eqref{cond-int-3-contrad}, we have
\[
\begin{split}
& E(t) \ge E(T_{14}) \, \text{Exp}\left[ \int_{T_{14}}^t \beta(s) \frac{\phi(N(s),N(s),\delta_2)}{\delta_2}p-\mu(s)-\epsilon(s) -\beta_S\kappa p \ds \right] \\
& = E(T_{14}) \, \text{Exp}\left[ \int_{T_{14}}^{T_{14}+\lambda\lfloor \frac{t-T_{14}}{\lambda}\rfloor} \beta(s) \frac{\phi(N(s),N(s),\delta_2)}{\delta_2}p-\mu(s)-\epsilon(s)-\beta_S\kappa p \ds \right. + \\
& \quad \quad \left. + \int_{T_{14}+\lambda\lfloor \frac{t-T_{14}}{\lambda}\rfloor}^t \beta(s) \frac{\phi(N(s),N(s),\delta_2)}{\delta_2}p-\mu(s)-\epsilon(s)-\beta_S\kappa p \ds \right]. \\
& \ge E(T_{14}) \ \text{Exp} \left[ (\eta-\beta_S\kappa p \lambda) \lfloor \frac{t-T_{14}}{\lambda} \rfloor -(\mu_S+\eps_S+\beta_S\kappa p)\lambda \right]\\
& \ge E(T_{14}) \ \text{Exp} \left[ \eta/2 \lfloor \frac{t-T_{14}}{\lambda} \rfloor -(\mu_S+\eps_S+\beta_S\kappa p)\lambda \right]
\end{split}
\]
and we conclude that $E(t) \to +\infty$. This is a contradiction with the boundedness of $E$ established in Proposition~\ref{subsection:BS}. We conclude that $\displaystyle \limsup_{t \to +\infty} I(t)> \eps_2$ holds.

Next we prove that
\begin{equation} \label{eq:liminf-ge-I1}
\liminf_{t \to +\infty} I(t) \ge \ell,
\end{equation}
where $\ell>0$ is some constant to be determined.

Similarly to~\eqref{cond-int-1-contrad}--\eqref{cond-int-3-contrad}, letting $\lambda_3=k\lambda>0$ with $k \in \N$ be sufficiently large and recalling~\eqref{eq:d-Lambda-}, we conclude that there is $T_{15} \ge T_{14}$, $\eps_1>0$, $\eps_2>0$, $\eps_3>0$ sufficiently small such that for all $t \ge T_{15}$ we have
\begin{equation}\label{liminf-dem0}
N(t)=S(t)+E(t)+R(t)+I(t)<2m_2,
\end{equation}
\begin{equation}\label{liminf-dem1}
\int_t^{t+\lambda_3} \beta(t)M\eps_2 - (\mu(s) + \epsilon(s)) \eps_1 \ds < -2m_2,
\end{equation}
\begin{equation}\label{liminf-dem2}
\int_t^{t+\lambda_3} \gamma(s) \eps_2 - (\mu(s)+\delta(s))\eps_3 \ds < -2m_2,
\end{equation}
\begin{equation}\label{liminf-dem3}
\int_t^{t+\lambda_3} \beta(s) \ \frac{\phi(N(s),N(s),\delta_2)}{\delta_2}p
-\mu(s)-\epsilon(s) \ds > k\eta,
\end{equation}
\[
\theta_1=\frac{m_1}{2}-\eps_1-[1+\beta_S M\lambda+\gamma_S\lambda]\eps_2-\eps_3>0.
\]
\begin{equation}\label{liminf-dem3b}
\kappa = K_{\theta_1}[\eps_1+[1+\beta_S M\lambda+\gamma_S\lambda]\eps_2+\eps_3] < \min \left\{ \frac{\eta}{2\beta_Sp\lambda}, \frac{2(\mu_S+\gamma_S)}{\beta_Sp} \right\}
\end{equation}

According to~\eqref{eq:limsupI>eps2} there are only two possibilities: there exists $T>0$ such that $I(t) \ge \eps_2$ for all $t \ge T$ or $I(t)$ oscillates about $\eps_2$.

In the first case we set $\ell=\eps_2$ and we obtain~\eqref{eq:liminf-ge-I1}.

Otherwise we must have the second case. Let $T_{17} \ge T_{16}>T_{15}$ be constants such that $W(p,t) \le 0$, for all $t \ge T_{15}$ (we may assume this by Lemma~\ref{lemma:tech}) and that $I(T_{16})=I(T_{17})=\eps_2$ and $I(t)<\eps_2$ for all $t \in [T_{16},T_{17}]$. Suppose first that $T_{17}-T_{16} \le C +2\lambda_3$ where $C$ satisfies
\begin{equation}\label{eq:def-C}
C \ge \frac{1}{\mu_S+\gamma_S} \, \left[(3\mu_S+\gamma_S+2\epsilon_S)\lambda_3+\ln \frac{2}{\eta k}\right],
\end{equation}
From the third equation in~\eqref{eq:ProblemaPrincipal} we have
\begin{equation}\label{eq:maj-dIdt}
I'(t) \ge -(\mu_S+\gamma_S)I(t).
\end{equation}
Therefore, we obtain for all $t \in [T_{16},T_{17}]$,
$$I(t) \ge I(T_{16}) \e^{-\int_{T_{16}}^t \mu_S+\gamma_S \ds} \ge \eps_2 \e^{-(\mu_S+\gamma_S)(C+2\lambda_3)}.$$
On the other hand, if $T_{17}-T_{16} > C +2\lambda_3$ then, from~\eqref{eq:maj-dIdt} we obtain
$$I(t) \ge \eps_2 \e^{-(\mu_S+\gamma_S)(C+2\lambda_3)},$$
for all $t \in [T_{16},T_{16}+C+2\lambda_3]$. Set $\ell=\eps_2 \e^{-(\mu_S+\gamma_S)(C+2\lambda_3)}$. We will show that $I(t) \ge \ell$ for all $t \in [T_{16}+C+2\lambda_3,T_{17}]$ and this establishes the result.

Suppose that $E(t) \ge \eps_1$ for all $t \in [T_{16},T_{16}+\lambda_3]$.
Then, from the second equation in~\eqref{eq:ProblemaPrincipal},~\eqref{cond-CC4},~\eqref{liminf-dem0} and~\eqref{liminf-dem1} we have
\[
\begin{split}
& \quad \ E(T_{16}+\lambda_3) \\
& = E(T_{16})+\int_{T_{16}}^{T_{16}+\lambda_3} \beta(s) \, \phi(S(s),N(s),I(s)) - (\mu(s)+\gamma(s)) E(t) \ds \\
& \le E(T_{16})+\int_{T_{16}}^{T_{16}+\lambda_3} \beta(s) M\eps_2 - (\mu(s)+\gamma(s)) \eps_1 \ds \\
& < 2m_2 - 2m_2=0,
\end{split}
\]
witch is a contradiction with~\ref{cond-1-bs}) in Proposition~\ref{subsection:BS}. Therefore, there exists a $T_{18} \in [T_{16},T_{16}+\lambda_3]$ such that $E(T_{18})<\eps_1$. Then, as in the proof of~\eqref{eq:boundE} and using~\eqref{liminf-dem1}, we can show that $E(t) \le \eps_1 + \beta_S M \eps_2 \lambda_3$, for all $t \ge T_{18}$. Also proceeding as in the proof of~\eqref{eq:boundR} and using~\eqref{liminf-dem2} we may assume that  $R(t) \le \eps_3 + \gamma_S \eps_2 \lambda_3$ for all $t \ge T_{18}$.

By~\eqref{eq:maj-dIdt} we have
\begin{equation} \label{eq:majI-mugamma}
I(t) \ge I(T_{16}) \e^{-\int_{T_{16}}^t \mu_S+\gamma_S \ds} = I(T_{16}) \e^{-(\mu_S+\gamma_S)(t-T_{16})} \ge
\eps_2 \e^{-(\mu_S+\gamma_S)\lambda_3}
\end{equation}
for all $t \in [T_{16}+\lambda_3,T_{16}+2\lambda_3]$.

Assume that there exists a $T_{19}>0$ such that $T_{19} \in [T_{16}+C+2\lambda_3,T_{17}]$, $I(T_{19}) = \ell$ and $I(t) \ge \ell$ for all $t \in [T_{16},T_{19}]$ (otherwise the result is established). By~\eqref{eq:s-to-z-1} and~\eqref{eq:majI-mugamma} we have, for all $t \in [T_{16}+\lambda_3,T_{16}+2\lambda_3]$,
\begin{equation}\label{eq:EEEEE}
\begin{split}
E'(t)
& \ge  \beta(t) \, \frac{\phi(S(t),N(t),\delta_2)}{\delta_2} I(t)-(\mu_S+\epsilon_S) E(t) \\
& \ge  \beta(t) \, \left( \frac{\phi(N(t),N(t),\delta_2)}{\delta_2}-\kappa \right) \eps_2 \e^{-(\mu_S+\gamma_S)\lambda_3} -(\mu_S+\epsilon_S) E(t),
\end{split}
\end{equation}
where $\kappa$ is given by~\eqref{cond-int-3-contrad}. By~\eqref{eq:EEEEE},~\eqref{liminf-dem3} and~\eqref{liminf-dem3b}, we get
\begin{equation}\label{eq:hihihi}
\begin{split}
&E(T_{16}+2\lambda_3) \\
& \ge e^{-(\mu_S+\epsilon_S)\lambda_3}E(T_{16}+\lambda_3) + \int_{T_{16}+\lambda_3}^{T_{16}+2\lambda_3} \beta(s) \,\left(\frac{ \phi(N(s),N(s),\delta_2)}{\delta_2}-\kappa \right) \eps_2 \times \\
& \quad \quad \times \e^{-(\mu_S+\gamma_S)\lambda_3}e^{-(\mu_S+\epsilon_S)(T_{16}+2\lambda_3-s)} \ds \\
& \ge e^{-(\mu_S+\gamma_S)\lambda_3} \int_{T_{16}+\lambda_3}^{T_{16}+2\lambda_3} \beta(s) \,\left(\frac{\phi(N(s),N(s),\delta_2)}{\delta_2}-\kappa \right) \eps_2 \e^{-(\mu_S+\epsilon_S)\lambda_3} \ds  \\
& \ge e^{-(2\mu_S+\gamma_S+\epsilon_S)\lambda_3} \eps_2 \int_{T_{16}+\lambda_3}^{T_{16}+2\lambda_3} \beta(s) \,\frac{\phi(N(s),N(s),\delta_2)}{\delta_2}-\beta_S \kappa \ds  \\
& \ge e^{-(2\mu_S+\gamma_S+\epsilon_S)\lambda_3} \eps_2 \, (k\eta/p-\beta_S \kappa \lambda_3)\\
& = e^{-(2\mu_S+\gamma_S+\epsilon_S)\lambda_3} \eps_2 \, (\eta/p-\beta_S \kappa \lambda) \, k \\
& > e^{-(2\mu_S+\gamma_S+\epsilon_S)\lambda_3} \eps_2 \eta k /(2p).
\end{split}
\end{equation}
On the other side, by \eqref{eq:majE-inf2} we obtain
\begin{equation}\label{eq:xpto}
  E'(t) \ge \left[ \beta(t) \, \frac{\phi(N(t),N(t),\delta_2)}{\delta_2}p
-\beta(t) \kappa p-\mu(t)-\epsilon(t)\right] E(t).
\end{equation}
and thus, by \eqref{eq:hihihi},~\eqref{liminf-dem3},~\eqref{liminf-dem3b} and~\eqref{eq:xpto}, letting $n=2+\lfloor \frac{T_{19}-T_{16}}{\lambda_3}\rfloor$
\[
\begin{split}
& \eps_2 \e^{-(\mu_S+\gamma_S)(C+2\lambda_3)}\\
& = I(T_{19}) \\
& \ge p E(T_{19})\\
& \ge p E(T_{16}+2\lambda_3) \text{Exp}\left[  \int_{T_{16}+2\lambda_3}^{T_{19}} \beta(s) \frac{\phi(N(s),N(s),\delta_2)}{\delta_2}p-\beta(s)\kappa p
-\mu(s)-\epsilon(s) \ds \right]\\
& \ge p E(T_{16}+2\lambda_3) \text{Exp}\left[  \int_{T_{16}+2\lambda_3}^{T_{16}+n\lambda_3} \beta(s) \frac{\phi(N(s),N(s),\delta_2)}{\delta_2}p-\beta(s)\kappa p
-\mu(s)-\epsilon(s) \ds \right. \\
& \quad \quad + \left. \int_{T_{16}+n\lambda_3}^{T_{19}} \beta(s) \frac{\phi(N(s),N(s),\delta_2)}{\delta_2}p-\beta(s)\kappa p
-\mu(s)-\epsilon(s) \ds\right]\\
& > pe^{-(3\mu_S+\gamma_S+2\epsilon_S)\lambda_3}\eps_2 \frac{\eta k}{2p} \e^{(n-2)(\eta k - \beta_S \kappa p \lambda_3)} \e^{-\beta_S \kappa p \lambda_3}\\
& > pe^{-(3\mu_S+\gamma_S+2\epsilon_S)\lambda_3}\eps_2 \frac{\eta k}{2p} \e^{(n-2)\eta k /2} \e^{-\beta_S \kappa p \lambda_3}\\
& > \frac12  e^{-(3\mu_S+\gamma_S+2\epsilon_S)\lambda_3}\eps_2 \eta k \e^{-\beta_S \kappa p \lambda_3}\\
& > \frac12 e^{-(3\mu_S+\gamma_S+2\epsilon_S)\lambda_3}\eps_2 \eta k \e^{-2(\mu_S+\gamma_S) \lambda_3}
\end{split}
\]
and this implies that
\[
C < \frac{1}{\mu_S+\gamma_S} \, \left[(3\mu_S+\gamma_S+2\epsilon_S)\lambda_3+\ln \frac{2}{\eta k}\right],
\]
contradicting~\eqref{eq:def-C}. This shows~\eqref{eq:liminf-ge-I1} and proves~\ref{teo:Permanence}. in the theorem.

We recall that, by~\eqref{eq:infboundmu}, there are $\mu_1,\mu_2>0$ sufficiently small and $T>0$ sufficiently large such that, for all $t \ge t_0 \ge T$ we have
$$-\int_{t_0}^t \mu(s) \ds \le -\mu_1 (t-t_0)+\mu_2.$$
Assume that $R_e(\lambda,p)<1$, $R_e^*(\lambda,p)<1$ and $G(p)<0$ and let $(S_1(t),0,0,R_1(t))$ be a disease-free solution of~\eqref{eq:ProblemaPrincipal} with $S_1(t_0)=S_{1,0}$ and $R_1(t_0)=R_{1,0}$ and let $(S(t),E(t),I(t),R(t))$ with $S(t_0)=S_0$, $E(t_0)=E_0$, $I(t_0)=I_0$ and $R(t_0)=R_0$ be some solution of~\eqref{eq:ProblemaPrincipal}.

Since we are in the conditions of~\ref{teo:Extinction}), for each $\eps>0$ there is $T_\eps>0$ such that $I(t) \le \eps$ for each $t \ge T_\eps$. Therefore, using the second equation in~\eqref{eq:ProblemaPrincipal}, we get, for $t \ge T_\eps$,
\[
\begin{split}
  E'(t)
  & = \beta(t) \frac{\phi(S(t),N(t),I(t))}{I(t)}I(t)-(\mu(t)+\epsilon(t))E(t) \\
  & \le \beta_S M \eps - \mu(t) E(t)
\end{split}
\]
and thus, for $t \ge t_0 \ge \max\{T,T_\eps\}$, we have
\[
\begin{split}
E(t)
& \le \e^{-\int_{t_0}^t \mu(s) \ds} E_0 + \int_{t_0}^t \beta_S M \eps \e^{-\int_u^t \mu(s) \ds} \du \\
& \le \e^{-\mu_1 (t-t_0)+\mu_2} E_0 + \beta_S M \eps \int_{t_0}^t \e^{-\mu_1 (t-u)+\mu_2} \du \\
& = \e^{-\mu_1 (t-t_0)+\mu_2} E_0 + \frac{\beta_S M \e^{\mu_2}}{\mu_1}  (1-\e^{-\mu_1(t-t_0)}) \eps
\end{split}
\]
and, since $\eps>0$ is arbitrary, we conclude that
\begin{equation}\label{eq:Eto0}
  \limsup_{t \to +\infty} E(t) = 0.
\end{equation}

By the fourth equation in~\eqref{eq:ProblemaPrincipal} and setting $w=R(t)-R_1(t)$, we have, for $t \ge T_\eps$
\[
\begin{split}
  w'(t)
  & = \gamma(t) I(t)-(\mu(t)+\eta(t))w(t) \\
  & \le \gamma_S \eps - (\mu(t)+\eta(t)) w(t)
\end{split}
\]
and thus, for $t \ge t_0 \ge \max\{T,T_\eps\}$, we have
\[
\begin{split}
w(t)
& \le \e^{-\int_{t_0}^t \mu(s)+\eta(s) \ds} (R_0-R_{0,1}) + \int_{t_0}^t \gamma_S \eps \e^{-\int_u^t \mu(s)+\eta(s) \ds} \du \\
& \le \e^{-\mu_1 (t-t_0)+\mu_2} (R_0-R_{0,1}) + \gamma_S \eps \int_{t_0}^t \e^{-\mu_1 (t-u)+\mu_2} \du \\
& = \e^{-\mu_1 (t-t_0)+\mu_2} (R_0-R_{0,1}) + \frac{\gamma_S \e^{\mu_2}}{\mu_1}  (1-\e^{-\mu_1(t-t_0)}) \eps
\end{split}
\]
and, since $\eps>0$ is arbitrary, we conclude that $\displaystyle \limsup_{t\to +\infty} R(t)-R_1(t) \le 0$. Repeating he computations with $w(t)$ replaced by $w_1(t)=R_1(t)-R(t)$ we conclude that $\displaystyle \limsup_{t\to +\infty} R_1(t)-R(t) \le 0$. Thus
\begin{equation}\label{eq:R-R0to0}
  \limsup_{t \to +\infty} |R(t)-R_1(t)| = \limsup_{t \to +\infty} |w(t)| = 0.
\end{equation}

Let $N=S+E+I+R$ and $N_1=S_1+R_1$ and set $u(t)=N(t)-N_1(t)$. By~\eqref{eq:SistemaAuxiliar}, for $t \ge T$,
we have $u'(t) = - \mu(t) u(t)$ and thus, for $t \ge t_0 \ge \max\{T,T_\eps\}$ we have
\[
u(t) = \e^{-\int_{t_0}^t \mu(s) \ds} (N_0-N_{0,1}) \le \e^{-\mu_1(t-t_0)+\mu_2} (N_0-N_{0,1}).
\]
Therefore
\begin{equation}\label{eq:S-S0to0}
\begin{split}
& \limsup_{t \to +\infty} |S(t)-S_1(t)| \\
& = \limsup_{t \to +\infty} |N(t)-E(t)-I(t)-R(t)-(N_1(t)-R_1(t))| \\
& \le \limsup_{t \to +\infty} \left(\, |N(t)-N_1(t)|+E(t)+I(t)+|R(t)-R_1(t)| \, \right)=0.
\end{split}
\end{equation}

By~\eqref{eq:Eto0},~\eqref{eq:R-R0to0} and~\eqref{eq:S-S0to0}, we have~\ref{teo:Extinction-behavior}. in the theorem.

%------------------------------------------------------------------------------
\subsection{Proof of Theorem~\ref{teo:structural_stability}.}
%------------------------------------------------------------------------------
Let $b_\delta^\tau$ denote the function in~\eqref{eq:bSlambda} with $\phi$ replaced by $\phi_\tau$.
Let $\delta>0$. We have that there is $L>0$ such that for $\tau \in [-L,L]$ we have by assumption $\sup_{t \ge 0} \left|\beta_\tau(t)-\beta(t)\right|<\delta$ and thus $\beta_\tau(t) < \beta_S + \delta$ for all $t\ge 0$. Write $B=\beta_S+\delta$. By~\eqref{eq:bSlambda} and~\eqref{cond-CC4} we have
\begin{equation} \label{eq:dif_b}
\begin{split}
  & |b^\tau_\delta(p,t,z(t))-b_\delta(p,t,z(t))| \\
  & = \left|\beta_\tau(s)\dfrac{\phi_\tau(z(t),z(t),\delta)}{\delta}p-\mu(t)-\epsilon_\tau(t)-\beta(s)\dfrac{\phi(z(t),z(t),\delta)}{\delta}p
  +\mu(t)+\epsilon(t) \right| \\
  & \le \left| \beta_\tau(t) \right| p \left| \dfrac{\phi_\tau(z(t),z(t),\delta)-\phi(z(t),z(t),\delta)}{\delta} \right| \\
  & \quad \quad + \left|\beta_\tau(t)-\beta(t)\right| p \dfrac{\phi(z(t),z(t),\delta)}{\delta}+\|\epsilon_\tau-\epsilon\|_\infty\\
  & \le B p \left| \dfrac{\phi_\tau(z(t),z(t),\delta)-\phi(z(t),z(t),\delta)}{\delta} \right|+
    Mp\|\beta_\tau-\beta\|_\infty+\|\epsilon_\tau-\epsilon\|_\infty
\end{split}
\end{equation}
Since for $\tau \in [-L,L]$, $\phi_\tau$ is differentiable and $\phi_\tau(x,y,0)=\phi(x,y,0)=0$, we get
\begin{equation}\label{eq:Taylor}
\begin{split}
& \left|\phi_\tau(z(t),z(t),\delta)-\phi(z(t),z(t),\delta)\right| \\
& = \left|\partial_3 \phi_\tau(z(t),z(t),0)\delta+R^\tau(\delta)- \partial_3 \phi(z(t),z(t),0)\right|\delta+R(\delta)|\\
& \le \left|\partial_3 \phi_\tau(z(t),z(t),0)- \partial_3 \phi(z(t),z(t),0)\right|\delta+|R^\tau(\delta)|+|R(\delta)|
\end{split}
\end{equation}
where $R(\delta)/\delta \to 0$ and $R^\tau(\delta)/\delta \to 0$ as $\delta \to 0$, where $\partial_3$ denotes the partial derivative with respect to the third coordinate. By~\eqref{eq:Taylor} we obtain
\begin{equation} \label{eq:diff_b_over_delta}
\begin{split}
& \dfrac{\left|\phi_\tau(z(t),z(t),\delta)-\phi(z(t),z(t),\delta)\right|}{\delta} \\
& \le \left|\partial_3 \phi_\tau(z(t),z(t),0)- \partial_3 \phi(z(t),z(t),0)\right|+\dfrac{|R^\tau(\delta)|}{\delta}+ \dfrac{|R(\delta)|}{\delta} \\
& \le \| \phi_\tau - \phi \|_{\Delta_{0,K}}+\dfrac{|R^\tau(\delta)|}{\delta}+ \dfrac{|R(\delta)|}{\delta}
\end{split}
\end{equation}
Thus, by~\eqref{eq:dif_b} and~\eqref{eq:diff_b_over_delta} we get
\[
\begin{split}
  & |b^\tau_\delta(p,s,z(s))-b_\delta(p,s,z(s))| \\
  & \le B p \left| \dfrac{\phi_\tau(z(t),z(t),\delta)-\phi(z(t),z(t),\delta)}{\delta} \right|+
    Mp\|\beta_\tau-\beta\|_\infty+\|\epsilon_\tau-\epsilon\|_\infty \\
  & \le B p \left( \| \phi_\tau - \phi \|_{\Delta_{0,K}}+\dfrac{|R^\tau(\delta)|}{\delta}+ \dfrac{|R(\delta)|}{\delta} \right)+
    Mp\|\beta_\tau-\beta\|_\infty+\|\epsilon_\tau-\epsilon\|_\infty.
\end{split}
\]
Therefore
\[
\begin{split}
 & \lim_{\delta \to 0^+} |b^\tau_\delta(p,s,z(s))-b_\delta(p,s,z(s))|\\
 & \quad \quad \quad \le B p \| \phi_\tau - \phi \|_{\Delta_{0,K}} + Mp\|\beta_\tau-\beta\|_\infty+\|\epsilon_\tau-\epsilon\|_\infty.
\end{split}
\]
Thus
\[
\begin{split}
& \left| \int_t^{t+\lambda} \lim_{\delta \to 0^+} b^\tau_\delta(p,s,z(s))- \lim_{\delta \to 0^+} b_\delta(p,s,z(s)) \ds \right| \\
& \le \int_t^{t+\lambda} \lim_{\delta \to 0^+} |b^\tau_\delta(p,s,z(s))-b_\delta(p,s,z(s))| \ds  \le \Theta(\tau),
\end{split}
\]
where
$$ \Theta(\tau) = \lambda B p \| \phi_\tau - \phi \|_{\Delta_{0,K}} + Mp\lambda\|\beta_\tau-\beta\|_\infty+\lambda\|\epsilon_\tau-\epsilon\|_\infty. $$
Thus
$$\ln R_e(\lambda,p) - \Theta(\tau) \le \ln R_e^\tau(\lambda,p) \le \ln R_e(\lambda,p) + \Theta(\tau)$$
and then
$$R_e(\lambda,p) \e^{-\Theta(\tau)} \le R_e^\tau(\lambda,p) \le R_e(\lambda,p) \e^{\Theta(\tau)}$$
and sending $\tau \to 0$ we get
$$\lim_{\tau \to 0} R_e^\tau(\lambda,p) = R_e(\lambda,p).$$
Similarly we obtain also $\displaystyle \lim_{\tau \to 0} \left(R_e^*\right)^\tau(\lambda,p) = (R_e^*)(\lambda,p)$, $\displaystyle \lim_{\tau \to 0} R_p^\tau(\lambda,p) = R_p(\lambda,p)$, $\displaystyle \lim_{\tau \to 0} \left(R_p^*\right)^\tau(\lambda,p) = (R_p^*)(\lambda,p)$, $\displaystyle \lim_{\tau \to 0} G^\tau(p) = G(p)$ and $\displaystyle \lim_{\tau \to 0} H^\tau(p) = H(p)$.
%-----------------------------------------------------------------------------%
\section{Discussion} \label{section:D}
%-----------------------------------------------------------------------------%

In this paper we considered a non-autonomous family of SEIRS models with general incidence and obtained conditions for strong persistence and extinction of the infectives. We obtained corollaries for autonomous and asymptotically autonomous systems, where the conditions became thresholds, and we obtained also corollaries for the general incidence periodic setting and for non-autonomous Michaelis-Menten incidence functions. To illustrate our results we considered some concrete family of periodic models and we obtained regions of strong persistence and extinction for several pairs of parameters.

Naturally we would like to obtain explicit thresholds for the general non-autonomous family. The regions obtained in figure~\ref{fig1} suggest that big oscillations in the parameters lead to situations where our conditions do not apply. This is a consequence of the use of $\limsup$ and $\liminf$ in conditions~\eqref{eq:relambdap} to~\eqref{eq:Hp}. We believe that to overcome this problem we must have expressions that include some features more closely linked to the shape of the incidence functions.

Finally, we saw that our conditions for strong persistence and extinction are robust in some general family of $C^1$ parameter functions. Naturally, if we restrict our family to the autonomous setting, this has to do with the fact that the thresholds are given by~\eqref{eq:RA} and it is immediate that small perturbations of the parameters in~\eqref{eq:RA} yield a number close to the original one.

To obtain Theorem~\ref{teo:structural_stability} we felt the need to assume that the birth and death rates remain the same for all the family. This motivates the following question: do we have the same result if we only assume that the birth and death rates are close in the $C^0$ topology?
%-----------------------------------------------------------------------------%
\bibliographystyle{elsart-num-sort}

\begin{thebibliography}{10}
\expandafter\ifx\csname url\endcsname\relax
  \def\url#1{\texttt{#1}}\fi
\expandafter\ifx\csname urlprefix\endcsname\relax\def\urlprefix{URL }\fi

\bibitem{Buonomo-Lacitignola-RM-2008} B. Buonomo and D. Lacitignola, On the dynamics of an SEIR epidemic model
with a convex incidence rate, Ricerche mat. 57, 261–281 (2008)

\bibitem{Chavez-Thieme-MPD-1995} C. Castillo-Chavez, H.R. Thieme, Asymptotically autonomous epidemic models, in: O. Arino, D. E. Axelrod, M. Kimmel, M. Langlais (Eds.), Mathematical Population Dynamics: Analisys and Heterogenity, Wuerz, Winnipeg, Canada, 1995, p. 33.

\bibitem{Driessche-Hethcote-JMB-1991} P. den Driessche, H. Hethcote, J. Math. Biol. 29, 271-287 (1991)

\bibitem{Driessche-Li-Muldowney-CAMQ-1999} P. van den Driessche, M. Li and J. Muldowney, Global Stability of SEIRS Models in Epidemiology, Canadian Applied Mathematics Quarterly 7, 409-425 (1999)

\bibitem{Kermack-McKendrick-PRSL-1927} W. O. Kermack and A. G. McKendrick,  A contribution to the mathematical theory of epidemics, Proc. R. Soc. Lond. A 115, 700--721 (1927)

\bibitem{Nakata-Kuniya-JMAA-2010} T. Kuniya and Y. Nakata, Global Dynamics of a class of SEIRS epidemic models in a periodic environement, J. Math. Anal. Appl. 363, 230-237 (2010)

\bibitem{Kuniya-Nakata-AMC-2012} T. Kuniya and Y. Nakata, Permanence and extinction for a nonautonomous SEIRS epidemic model, Appl. Math. Computing 218, 9321-9331 (2012)

\bibitem{Li-Zhou-CSF-2009} X. Li, L. Zhou, Global Stabiliy of an SEIR Epidemic model with vertical transmission and saturating contact rate, Chaos, Solitons and Fractals 40, 874-884 (2009)

\bibitem{Markus-CTNO-1956} L. Markus, Asymptotically autonomous differential systems, Contributions to the theory of Nonlinear Oscillations 111 (S. Lefschetz, ed.), Ann. Math. Stud., 36, Princeton University Press, Princeton, NJ, 17-29 (1956)

\bibitem{Mischaikow-Smith-Thieme-TAMS-1995} K. Mischaikow, H. Smith and H. R. Thieme, Asymptotically autonomous semiflows: chain recurrence and Lyapunov functions, Trans. Amer. Math. Soc., 347, 1669-1685 (1995)

\bibitem{Pereira-Silva-Silva-MMAS-2013} E. Pereira, C. M. Silva and J. A. L. Silva, A Generalized Non-Autonomous SIRVS Model, Math. Meth. Appl. Sci. 36, 275-289 (2013)

\bibitem{Rebelo-Margheri-Bacaër-2012} C. Rebelo, A. Margheri and N. Baca\"er, Peristence in seasonally forced epidemiological models, J. Math. Biol. 54, 933-949 (2012)

\bibitem{Shope-EHP-1991} R. Shope, Environemental Health Perspectives 96, 171-174 (1991)

\bibitem{Zhang-Teng-BMB-2007}  T. Zhang and Z. Teng, On a nonautonomous SEIRS model in epidemiology, Bull. Math. Biol. 69, 2537-2559 (2007)

\bibitem{Wang-Zhao-JDDE-2008} W. Wang and X.-Q. Zhao, Threshold dynamics for compartmental epidemic models in periodic environments, J. Dyn. Diff. Equat., 20, 699-717 (2008)

\bibitem{Zhang-Teng-Gao-AA-2008} T. Zhang, Z. Teng and S. Gao, Threshold conditions for a nonautonomous epidemic model with vaccination, Applicable Analysis, 87, 181-199 (2008)

\bibitem{Zhang-Yingqi-Xu-AM-2013} H. Zhang, L. Yingqi, W. Xu, Global stability of an SEIS epidemic model with general saturation incidence, Appl. Math., Art. ID 710643 (2013)

\end{thebibliography}

%-----------------------------------------------------------------------------%
%-----------------------------------------------------------------------------%
\end{document}